\pdfoutput=1

\documentclass{amsart}
\usepackage[utf8]{inputenc}
\usepackage[T1]{fontenc}

\usepackage{amssymb}
\usepackage{braket}
\usepackage{mathrsfs}
\usepackage{ifthen}
\usepackage{here}
\usepackage{todonotes}
\usepackage{tikz}
\usepackage{mathtools}
\usetikzlibrary{patterns,decorations.pathreplacing,calligraphy}
\usepackage{comment} 
\usepackage{mleftright}

\usepackage{amsmath,amsthm,verbatim,amsfonts, graphicx,booktabs}

\usepackage[dvipsnames]{xcolor}
\usepackage[pagebackref,hypertexnames=false,colorlinks=true,
            linkcolor=NavyBlue,
            citecolor=NavyBlue,
            urlcolor=NavyBlue]{hyperref} 
\usepackage{cleveref}
 \usepackage[all]{xy}
 \usepackage{amscd}
 \usepackage[alphabetic,backrefs,msc-links]{amsrefs}
 \usepackage{color}
 \usepackage{enumitem}
\newlist{steps}{enumerate}{1}
\setlist[steps, 1]{label = Step \arabic*:}
\usepackage[abs]{overpic}
\usepackage{tikz-cd}
\usepackage{pinlabel}

\usepackage{geometry}
\makeatletter
\DeclareRobustCommand\widecheck[1]{{\mathpalette\@widecheck{#1}}}
\def\@widecheck#1#2{%
   \setbox\z@\hbox{\m@th$#1#2$}%
   \setbox\tw@\hbox{\m@th$#1%
      {%
         \vrule\@width\z@\@height\ht\z@
         \vrule\@height\z@\@width\wd\z@}$}%
   \dp\tw@-\ht\z@
   \@tempdima\ht\z@ \advance\@tempdima2\ht\tw@ \divide\@tempdima\thr@@
   \setbox\tw@\hbox{%
      \raise\@tempdima\hbox{\scalebox{1}[-1]{\lower\@tempdima\box\tw@}}}%
   {\ooalign{\box\tw@ \cr \box\z@}}}
   \let\@wraptoccontribs\wraptoccontribs
\makeatother

\theoremstyle{plain}
\newtheorem*{theorem*}{Theorem}
\newtheorem{thm}{Theorem}[section]
\crefname{thm}{Theorem}{Theorems}
\Crefname{thm}{Theorem}{Theorems}
\newtheorem{prop}[thm]{Proposition}
\crefname{prop}{Proposition}{Propositions}
\Crefname{prop}{Proposition}{Propositions}
\newtheorem{lem}[thm]{Lemma}
\crefname{lem}{Lemma}{Lemmas}
\Crefname{lem}{Lemma}{Lemmas}
\newtheorem{cor}[thm]{Corollary}
\crefname{cor}{Corollary}{Corollaries}
\Crefname{cor}{Corollary}{Corollaries}
\newtheorem{rem}[thm]{Remark}
\crefname{rem}{Remark}{Remarks}
\Crefname{rem}{Remark}{Remarks}

\crefname{claim}{Claim}{Claims}
\Crefname{claim}{Claim}{Claims}

\crefname{property}{Property}{Properties}
\Crefname{property}{Property}{Properties}

\crefname{problem}{Problem}{Problems}
\Crefname{problem}{Problem}{Problems}

\crefname{conjecture}{Conjecture}{Conjecture}
\Crefname{conjecture}{Conjecture}{Conjecture}

\theoremstyle{definition}
\newtheorem{defn}[thm]{Definition}
\crefname{defn}{Definition}{Definitions}
\Crefname{defn}{Definition}{Definitions}

\crefname{notation}{Notation}{Notations}
\Crefname{notation}{Notation}{Notations}

\crefname{convention}{Convention}{Conventions}
\Crefname{convention}{Convention}{Conventions}
\crefname{cond}{Condition}{Conditions}
\Crefname{cond}{Condition}{Conditions}

\crefname{assum}{Assumption}{Assumptions}
\Crefname{assum}{Assumption}{Assumptions}

\Crefname{ques}{Question}{Question}

\theoremstyle{remark}
\crefname{rem}{Remark}{Remarks}
\Crefname{rem}{Remark}{Remarks}

\crefname{ex}{Example}{Examples}
\Crefname{ex}{Example}{Examples}

\crefname{section}{Section}{Sections}
\Crefname{section}{Section}{Sections}
\crefname{subsection}{Subsection}{Subsections}
\Crefname{subsection}{Subsection}{Subsections}
\crefname{figure}{Figure}{Figures}
\Crefname{figure}{Figure}{Figures}

\AddToHook{env/thm/begin}{\crefalias{thm}{theorem}}
\AddToHook{env/defn/begin}{\crefalias{thm}{definition}}
\AddToHook{env/lem/begin}{\crefalias{thm}{lemma}}
\AddToHook{env/prop/begin}{\crefalias{thm}{proposition}}
\AddToHook{env/rem/begin}{\crefalias{thm}{remark}}
\AddToHook{env/ex/begin}{\crefalias{thm}{example}}

\newcommand{\bt}{{\bf t}}

\def\sf{\text{sf}}

\def\bt{{\mathbf t}}

\def\tt{_{T,\bt}}

\newcommand{\mbar}[1]{{\ooalign{\hfil#1\hfil\crcr\raise.167ex\hbox{--}}}}

     \RequirePackage{rotating}                   
    \def\HMt{%
       \setbox0=\hbox{$\widehat{\mathit{HM}}$}
       \setbox1=\hbox{$\mathit{HM}$}
       \dimen0=1.1\ht0
       \advance\dimen0 by 1.17\ht1
       \smash{\mskip2mu\raise\dimen0\rlap{%
          \begin{turn}{180}
              {$\widehat{\phantom{\mathit{HM}}}$}
           \end{turn}} \mskip-2mu    
                \mathit{HM}
                    }{\vphantom{\widehat{\mathit{HM}}}}{}}

\title[Local Knots, $\nu^+$-Sharp Knots, and Rational Slice Genus]{Local Knots, $\nu^+$-Sharp Knots, and Rational Slice Genus}

\author{Junghwan Park}
\address{Department of Mathematical Sciences, KAIST, Republic of Korea}
\email{jungpark0817@kaist.ac.kr}

\author{Zhongtao Wu}
\address{Department of Mathematics, The Chinese University of Hong Kong, Hong Kong}
\email{ztwu@math.cuhk.edu.hk}
\author{Jingling Yang}
\address{School of Mathematical Sciences, Xiamen University, China}
\email{yangjingling@xmu.edu.cn}

\begin{document}



\begin{abstract}
Hom and Wu introduced the knot concordance invariant $\nu^{+}$ for knots in $S^{3}$ and proved that it gives a lower bound for the slice genus. Wu and Yang extended $\nu^{+}$ to knots in rational homology $3$-spheres, where it gives a lower bound for the rational slice genus, an analogue of the slice genus for knots in rational homology $3$-spheres. We call a knot $\nu^{+}$-sharp if this bound is realized as an equality.

An open question asks whether a local knot in a $3$-manifold $Y$, that is, a knot contained in a $3$-ball, can bound a surface of smaller genus in $Y\times I$ than in $S^{3}\times I$. Using the Heegaard Floer invariant $\nu^+$, we show that this does not occur for local knots arising from $\nu^+$-sharp knots: if $K\subset S^3$ is $\nu^+$-sharp and $Y$ is a rational homology $3$-sphere, then the induced local knot in $Y$ has rational slice genus equal to the slice genus of $K$. The proof proceeds by establishing an additivity result for the rational slice genus.
\end{abstract}

\maketitle
\section{Introduction}\label{sec:intro}
Given a knot $K$ in the $3$-sphere $S^3$, the \emph{slice genus} of $K$ is defined by
\[
g_4(K):=\min_{F} g(F),
\]
where the minimum is taken over all compact, oriented, connected surfaces $F$ that are smoothly and properly embedded in $S^3 \times I$ and satisfy
$\partial F=K \subset S^3 \times \{1\}$.
 The slice genus has been studied extensively; see, for example,~\cite{Murasugi:1965, Kronheimer-Mrowka:1993,OSf,Rasmussen:2010,HW}. Many lower bounds for $g_4(K)$ have been obtained using a variety of techniques. One of the most useful lower bounds comes from Heegaard Floer theory~\cite{OzSz:2004, OSk, Rasmussen:2003} and is given by the knot concordance invariant $\nu^+(K)$, introduced by Hom and the second author~\cite{HW}, who proved that
\[
\max \left\{\nu^+(K), \nu^+(-K) \right\} \leq g_4(K).
\]
We say that $K$ is \emph{$\nu^+$-sharp} if the above inequality is an equality. Here, $-K$ represents the reverse of the mirror of $K$ in $S^3$.

Analogously, for a knot $K$ in a rational homology $3$-sphere $Y$, one can define the \emph{rational slice genus}
(more precisely, the \emph{rational slice genus relative to the rational longitude}) as follows:
\begin{equation*}
\label{Qslicegenus}
g_{Y \times I}(K):=\mathop{\min}\limits_{F} \dfrac{-\chi(F)}{2|[\mu] \cdot [\partial F]|}+\frac{1}{2},
\end{equation*}
where the minimum is taken over all Seifert-framed rational slice surfaces $F$ for $K$, and $\mu$ denotes a meridian of $K$ (see Section~\ref{Rational slice genus} for the precise definition).\footnote{We normalize the rational slice genus
in~\cite{WY} by adding $\frac{1}{2}$ so that this definition agrees with the slice genus when $K$ is a knot in $S^3$. We also remark that there are other variants of the definition of the rational slice genus; see, for example,~\cite{Raoux-Hedden:2023}.}
A systematic study of this notion was initiated in~\cite{WY} using Heegaard Floer theory. Among other things,
it generalized the $\nu^+$-invariant to knots in rational homology $3$-spheres, denoted by $\nu^+(Y,K)$, and proved that
\[
\max \left\{\nu^+(Y,K), \nu^+(-Y,-K) \right\} \leq g_{Y\times I}(K).
\]
Here, $-K \subset -Y$ denotes the knot obtained by reversing the orientations of both $K$ and $Y$. The notion of \emph{$\nu^+$-sharp} naturally generalizes to knots in rational homology $3$-spheres, namely, knots  for which the above inequality is an equality.



\subsection*{Rational slice genus of local knots}
Our main motivation for this article is to understand the rational slice genus of local knots in rational homology $3$-spheres. A knot $K$ in a $3$-manifold $Y$ is called \emph{local} if it is contained in a $3$-ball in $Y$. By definition, for a local knot in a rational homology $3$-sphere $Y$, and more generally for a knot $K$ that is null-homologous in $Y$, we have
\[
g_{Y \times I}(K)=\min_{F} g(F),
\]
where the minimum is taken over all compact, oriented, connected surfaces $F$ that are smoothly and properly embedded in $Y \times I$ and satisfy
$\partial F=K \subset Y \times \{1\}$.

Note that each local knot $K\subset Y$ is induced by a knot $K\subset S^3$. When there is no ambiguity, we will abuse notation and use the same symbol for both. With this convention, we clearly have
\[
g_{Y \times I}(K) \leq g_4(K)
\]
for any local knot $K\subset Y$. An interesting open question is whether $g_{Y \times I}(K)=g_4(K)$ for all local knots (see, e.g.,~\cite[Section~2]{KR}). If true, this would mean that allowing a larger ambient $3$-manifold does not permit a more efficient choice of bounding surface, which is perhaps somewhat counterintuitive. Note that this question is not restricted to knots in rational homology $3$-spheres. More generally, for any local knot $K$ in an arbitrary $3$-manifold $Y$, we may define $g_{Y\times I}(K)$ analogously.

There are several known cases in which
\[
g_{Y\times I}(K)=g_4(K).
\]
In~\cite[Proposition~2.9]{NOPP} (see also~\cite[Proposition~2.2]{KR}), the authors show that this equality always holds when $g_4(K)\leq 1$. In particular, if there exists an example for which the inequality is strict, then it must satisfy $g_4(K)>1$. Moreover,~\cite[Theorem~2.5]{DNPR} proves that the equality holds for every local knot in $S^1\times S^2$, and~\cite[Proposition~2.4]{KR} proves that it also holds for every local knot in a $3$-manifold that smoothly embeds in $S^4$.

We extend these results by showing that $\nu^+$-sharp knots provide a broad new class of local knots for which the two genera agree.

\begin{thm}\label{thm:localknotgenus-S3}
Let $K$ be a $\nu^+$-sharp knot in $S^3$ and let $Y$ be a rational homology $3$-sphere. If $K\subset Y$ is the induced local knot, then it is $\nu^+$-sharp and
$g_{Y \times I}(K)=g_4(K)$.
\end{thm}

More generally, we show that the same phenomenon holds for $\nu^+$-sharp knots in arbitrary rational homology $3$-spheres.


\begin{thm}\label{thm:localknotgenus}
Let $K$ be a $\nu^+$-sharp knot in a rational homology $3$-sphere $Y$, and let $U$ be the unknot in a rational homology $3$-sphere $Y'$. Then $K\# U \subset Y\# Y'$ is $\nu^+$-sharp and
\[
g_{(Y\#Y') \times I}(K\# U)
= g_{Y \times I}(K).
\]
\end{thm}

\subsection*{Additivity of rational slice genus}
The two theorems above are consequences of a more general result concerning the additivity of the rational slice genus.

Let $K$ and $K'$ be knots in rational homology $3$--spheres $Y$ and $Y'$, respectively. For the rational slice genus, we have the usual subadditivity (see Lemma~\ref{lem:subadditivity}):
\[
g_{(Y\#Y') \times I}(K\#K') \leq g_{Y \times I}(K) + g_{Y' \times I}(K').
\]
This inequality can be strict, while the $3$--dimensional analogue is additive; see~\cite[Lemma~5.1]{NV}.

We prove that if $K$ is $\nu^+$-sharp and $K'$ satisfies a suitable condition, then additivity holds for the rational slice genus. Roughly speaking, the condition on $K'$ is that, for each $\rm Spin^c$ structure, its knot Floer complex splits, up to acyclic summands, as a shifted copy of the unknot complex. We call such knots \emph{totally locally trivial}; see Definition~\ref{def:localtrivial}.


\begin{thm}\label{KcUY1}
Let $K$ and $K'$ be $\nu^+$-sharp knots in rational homology $3$-spheres $Y$ and $Y'$ respectively. If, in addition, $K'$ is totally locally trivial, then $K\#K'$ is $\nu^+$-sharp and
\[
g_{(Y\# Y')\times I}(K\# K') = g_{Y\times I}(K) + g_{Y'\times I}(K').
\]
\end{thm}

We note that some hypothesis on $K'$ is necessary: if $K\subset S^3$ has positive slice genus and we take $K'=-K\subset S^3$, then additivity fails. Moreover, requiring $K'$ to be merely $\nu^+$-sharp is still not sufficient. Indeed, there are $\nu^+$-sharp knots $K\subset S^3$, arising as connected sums of two torus knots~\cite{BCG:2017, Feller-Park:2021}, for which both $K$ and $-K$ are $\nu^+$-sharp.

In Lemma~\ref{UinQHS}, we prove that the unknot in a rational homology $3$-sphere is totally locally trivial and $\nu^+$-sharp. Thus, Theorems~\ref{thm:localknotgenus-S3} and~\ref{thm:localknotgenus} follow directly by taking $K'=U$ in Theorem~\ref{KcUY1}.

Using Theorem~\ref{KcUY1}, we can also consider another common class of totally locally trivial knots, namely \emph{Floer simple knots}. Simple knots were introduced by Berge in his study of lens space surgeries~\cite{Berge}. Recall that a \emph{lens space} is a closed, oriented $3$-manifold other than $S^3$ and $S^1\times S^2$ admitting a genus $1$ Heegaard splitting. A pair of compressing disks, one in each Heegaard solid torus, is called \emph{standard} if their boundary circles are in minimal position on the Heegaard torus. A \emph{simple knot} in a lens space is then obtained by taking the union of two properly embedded arcs, one in each of these standard compressing disks. Moreover, for each homology class in a lens space, one can explicitly construct a simple knot representing it.

Heegaard Floer theory associates to each closed $3$-manifold $Y$ the group $\widehat{HF}(Y)$, and to a knot $K\subset Y$ the group $\widehat{HFK}(K)$. When $Y$ is a rational homology $3$-sphere, their ranks satisfy
\[
\mathrm{rk}\,\widehat{HFK}(K) \ge \mathrm{rk}\,\widehat{HF}(Y) \ge \lvert H_1(Y;\mathbb{Z})\rvert.
\]
We say that $Y$ is an \emph{$L$-space} if the second inequality is an equality. A knot in an $L$-space is called \emph{Floer simple} if, furthermore, the first inequality is an equality. In particular, lens spaces are $L$-spaces, and simple knots in lens spaces provide fundamental examples of Floer simple knots. More examples of Floer simple knots are given in Section~\ref{Floer simple knots}.

For Floer simple knots, $\nu^+$ detects an analogue of the Seifert genus for knots in rational homology $3$-spheres, namely the rational Seifert genus. More precisely, the \emph{rational Seifert genus} of $K$ is defined as
\begin{equation*}\label{Qseifertgenus}
g_{Y}(K):=\min_{S} \dfrac{-\chi(S)}{2|[\mu] \cdot [\partial S]|} + \frac{1}{2},
\end{equation*}
where the minimum is taken over all rational Seifert surfaces $S$ for $K$ (see Section~\ref{Rational slice genus} for the precise definition).\footnote{We normalize the rational Seifert genus in~\cite{NW0, WY} by adding $\frac{1}{2}$ so that this definition agrees with the Seifert genus when $K$ is a knot in $S^3$.}
For a Floer simple knot $K$ in  $Y$, we have
\[
\nu^+(Y,K)=g_{Y \times I}(K)=g_Y(K).
\]
See~\cite[Proposition~5.1]{NW0} and~\cite[Theorems~1.1, 1.3, and~1.4]{WY}.

Taking $K'$ to be a Floer simple knot in Theorem~\ref{KcUY1}, we immediately obtain the following:
\begin{cor}\label{4gad}
Let $K$ be a $\nu^+$-sharp knot in a rational homology $3$-sphere $Y$. If $K'$ is a Floer simple knot in an $L$-space $Y'$, then $K\#K'$ is $\nu^+$-sharp and
\begin{equation*}\pushQED{\qed}
g_{(Y\#Y') \times I}(K\#K')
= g_{Y \times I}(K)+ g_{Y' \times I}(K')
= g_{Y \times I}(K)+ g_{Y'}(K').
\end{equation*}
Moreover, if $K\subset Y$ is a local knot, then
\begin{equation*}
g_{(Y\#Y') \times I}(K\#K')
= g_{Y \times I}(K)+ g_{Y' \times I}(K')
= g_{4}(K)+ g_{Y'}(K').\qedhere
\end{equation*}
\end{cor}

With this corollary, it is possible to compute the rational slice genera of various knots. For instance, the rational slice genus of the knot in Figure~\ref{UknotH} can be determined, where we perform $n$-surgery along an $L$-space knot $J\subset S^3$ for sufficiently large $n$, and $K$ is a $\nu^+$-sharp knot in $S^3$.

\begin{figure}[H]
\centering
\includegraphics[width=0.3\textwidth]{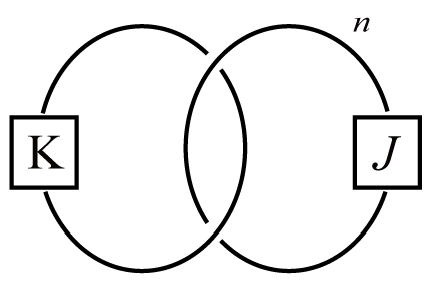}
\caption{A knot for which the rational slice genus can be computed using Corollary~\ref{4gad}.}
\label{UknotH}
\end{figure}

Note that when $n>2g(J)-1$, $n$-surgery on $J$ yields an $L$-space, and the meridian of $J$ in the resulting $L$-space is Floer simple~\cite{Gre15,Ras07,Hed11}, where $g(J)$ denotes the Seifert genus of $J$. More examples will be presented in Section~\ref{sec:examples_nu_sharp_floer_simple}, including explicit families of $\nu^+$-sharp knots and Floer simple knots.

The proof of Theorem~\ref{KcUY1} is based on Proposition~\ref{nuadditivity}, which establishes an additivity result for the $\nu^+$-invariant. In general, as in the case of knots in $S^3$, the $\nu^+$-invariant is only subadditive; see Theorem~\ref{thm:nuSubadditivity}. In Proposition~\ref{nuadditivity}, we give a necessary and sufficient condition for a knot to satisfy $\nu^+$-additivity when taking connected sums with arbitrary knots. Both the additivity and subadditivity results for $\nu^+$ may be of independent interest.

\subsection*{Organization}
The paper is organized as follows. In Section~\ref{Preliminaries}, we review some preliminaries, including the definition of the rational slice genus and background on knot Floer homology. In Section~\ref{(Sub)additivities}, we prove the subadditivity and additivity results for the $\nu^+$-invariant. In Section~\ref{Proof of main results}, we prove our main results. Section~\ref{sec:examples_nu_sharp_floer_simple} presents examples of $\nu^+$-sharp knots. Throughout, we work over $\mathbb{F}=\mathbb{Z}/2\mathbb{Z}$.

\subsection*{Acknowledgements}
We would like to thank Zhechi Chen, Stefan Friedl, Matthew Hedden, and Jennifer Hom for helpful discussions and suggestions. The first author is partially supported by the Samsung Science and Technology Foundation (SSTF-BA2102-02) and by the NRF grant RS-2025-00542968. The second author is partially supported by a grant from the Research Grants Council of Hong Kong Special Administrative Region, China (Project No. 14301825).
The third author is supported by the National Natural Science Foundation of China (Project No. 12301087), Fujian Provincial Natural Science Foundation of China (Project No. 2024J08012), and Fundamental Research Funds for the Central Universities (Project No. 20720230026).

We recently learned that Stefan Friedl, Tejas Kalelkar, Jos{\'e} Pedro Quintanilha, and Tanushree Shah have an upcoming paper on an analogous problem. Whereas our paper focuses on the slice genus of local knots, theirs investigates the unknotting number and the Gordian distance~\cite{FKQS2026}.

\section{Preliminaries}\label{Preliminaries}

\subsection{Rational Seifert surfaces and slice surfaces}\label{Rational slice genus}

We first recall the definitions related to the rational longitude; see~\cite[Section~1]{WY} for a more detailed discussion.
Let $K \subset Y$ be a knot in a rational homology $3$-sphere, and set
\[
M = Y \smallsetminus N^\circ(K),
\]
where $N(K)$ denotes a tubular neighborhood of $K$ and $N^\circ(K)$ its interior.
Note that the inclusion map
\[
i_{\ast}\colon H_{1}(\partial M;\mathbb{Z}) \rightarrow H_{1}(M;\mathbb{Z})
\]
has kernel isomorphic to $\mathbb{Z}$.
Thus there exists a primitive class $\lambda_r\in H_{1}(\partial M;\mathbb{Z})$ and a positive integer $k$ such that
\[
\ker(i_\ast)=\langle k\lambda_r\rangle .
\]
The class $\lambda_r$ determines a well-defined slope on $\partial M$, called the \textit{rational longitude} of $K$.

A \textit{rational Seifert surface} for $K$ is a properly embedded, compact, connected, oriented surface $S\subset M$ such that
\[
[\partial S]=k\lambda_r\in H_{1}(\partial M;\mathbb{Z}).
\]

A \textit{Seifert framed rational slice surface} for $K$ is a compact, connected, oriented surface $F$ smoothly embedded in
\[
(Y\times[0,1])\smallsetminus N^{\circ}(K)
\]
such that $\partial F=F\cap \partial N(K)$ and
\[
[\partial F]=k\lambda_r\in H_{1}(\partial N(K);\mathbb{Z}),
\]
where $N(K)$ denotes a solid torus neighborhood of $K$ in $Y \times \{1\}$. We call $F$ \textit{Seifert framed} because its boundary has the same slope as a rational Seifert surface for $K$.




\subsection{Knot Floer complexes}
Let $K$ be a knot in a rational homology $3$-sphere $Y$. Fix a doubly pointed Heegaard diagram $(\Sigma,\alpha,\beta,w,z)$ for $(Y,K)$ and a $\rm Spin^c$ structure $\mathfrak{s}$ on $Y$. The knot Floer complex $C_{\mathfrak{s}}=CFK^{\infty}(Y,K,\mathfrak{s})$ is generated by triples $[x,i,j]$, where $x\in \mathbb{T}_{\alpha}\cap \mathbb{T}_{\beta}$ satisfies $\mathfrak{s}_{w}(x)=\mathfrak{s}$, $i\in \mathbb{Z}$, and
\[
j-i=A_{Y,K}\bigl(\underline{\mathfrak{s}}_{w,z}(x)\bigr).
\]
Here $A_{Y,K}\colon \underline{\rm Spin^c}(Y,K)\to \mathbb{Q}$ denotes the Alexander grading on relative ${\rm Spin^c}$ structures associated to $(Y,K)$. The differential is the usual one defined by counting holomorphic disks. See~\cite[Section~2.4]{HL} for more details.

Let $\underline{\rm Spin^c}(Y,K,\mathfrak{s})$ denote the set of relative $\rm Spin^c$ structures with underlying ${\rm Spin^c}$ structure $\mathfrak{s}$, that is, those $\xi$ with $G_{Y,K}(\xi)=\mathfrak{s}$. For $\xi\in \underline{\rm Spin^c}(Y,K,\mathfrak{s})$, define
\[
A^-_{\xi}(K)=C_{\mathfrak{s}}\{\max\{\,i,\,j-A_{Y,K}(\xi)\,\}\leq 0\}
\qquad\text{ and }\qquad
B^-_{\xi}(K)=C_{\mathfrak{s}}\{i\leq 0\}.
\]
In particular,
\[
B^-_{\xi}(K)=C_{\mathfrak{s}}\{i\leq 0\}\cong CF^-(Y,\mathfrak{s}),
\]
and $A^-_{\xi}(K)$ is quasi-isomorphic to the complex $CF^-$ of a sufficiently large surgery along $K$ in a suitable ${\rm Spin^c}$ structure.
 
\subsection{$V$-invariants}

For a knot $K\subset S^3$, Ni and the second author defined a family of $\mathbb{Z}_{\geq 0}$-valued invariants $\{V_k(K)\}_{k\in \mathbb{Z}}$ in~\cite{NW1}. The analogous $V$-invariants for a knot $K$ in a rational homology $3$-sphere $Y$ can be defined in a similar manner as follows; see~\cite[Section~5.1]{WY}.

Let $v^-_{\xi}\colon A^-_{\xi}(K)\to B^-_{\xi}(K)$ be the inclusion map. It induces
\[
v^-_{\xi,\ast}\colon H_{\ast}(A^-_{\xi}(K)) \longrightarrow H_{\ast}(B^-_{\xi}(K)) \cong HF^-(Y,\mathfrak{s}).
\]
Both $H_{\ast}(A^-_{\xi}(K))$ and $H_{\ast}(B^-_{\xi}(K))$ are isomorphic to the direct sum of $\mathbb{F}[U]$ and a finite-dimensional $U$-torsion module. The map $v^-_{\xi,\ast}$ induces a homogeneous, non-zero map between the free parts $\mathbb{F}[U]$, which is necessarily multiplication by $U^{N}$ for some non-negative integer $N$. We define
\[
V_{\xi}(Y,K):=N.
\]

The \textit{correction term} $d(Y,\mathfrak{s})$ is defined as the maximum Maslov grading of any non-torsion element in $HF^-(Y,\mathfrak{s})$. Since $U$ decreases the Maslov grading by two, we can reformulate the definition of $V$-invariants as follows:
\[
V_{\xi}(Y,K):=\frac{1}{2}\Bigl(d(Y,\mathfrak{s})-\max\bigl\{{\rm gr}(x)\colon x\in H_{\ast}(A^-_{\xi}(K))\ \textit{\rm is homogeneous and non-torsion}\bigr\}\Bigr),
\]
where ${\rm gr}$ denotes the Maslov grading and $\mathfrak{s}=G_{Y,K}(\xi)$.





\subsection{$\nu^+$-invariants}

The \textit{$\nu^+$-invariant} of a knot $K\subset S^3$ is defined by
\[
\nu^+(K):=\min\{k\in \mathbb{Z}\mid V_k(K)=0\}.
\]
We define the $\nu^+$-invariants of knots in rational homology $3$-spheres similarly.

\begin{defn}[{\cite[Definition~2.7]{WY}}]\label{nudef}
Let $K$ be a knot in a rational homology $3$-sphere $Y$. Given $\mathfrak{s}\in {\rm Spin^c}(Y)$, define
\[
\nu^+_{\mathfrak{s}}(Y,K):=\min\{A_{Y,K}(\xi)\mid \xi\in \underline{\rm Spin^c}(Y,K,\mathfrak{s})\ \textit{\rm and}\ V_{\xi}(Y,K)=0\}
\]
and
\[
\nu^+(Y,K):=\mathop{\max}\limits_{\mathfrak{s}\in {\rm Spin^c}(Y)} \nu^+_{\mathfrak{s}}(Y,K).
\]
\end{defn}

\noindent When $Y=S^3$, the Alexander grading $A_{Y,K}$ identifies $\underline{\rm Spin^c}(S^3,K)$ with $\mathbb{Z}$ and recovers the usual indexing of the invariants $V_k(K)$.

The $\nu^+$-invariant gives a lower bound on the rational slice genus.

\begin{thm}[{\cite[Theorem~1.4]{WY}}]\label{nuleqprsg}
For any knot $K$ in a rational homology $3$-sphere $Y$, we have
\[
\nu^+(Y,K)\leq g_{Y\times I}(K).
\]
In particular, Floer simple knots attain equality.
\end{thm}

\begin{rem}
Since $g_{Y\times I}(K)=g_{(-Y)\times I}(-K)$, we obtain the following bound on the rational slice genus:
\[
\max\{\nu^+(Y,K),\,\nu^+(-Y,-K)\}\leq g_{Y\times I}(K).
\]
\end{rem}

\subsection{Symmetries in knot Floer homology and the middle relative $\rm Spin^c$ structure}

In this subsection, we record several key facts in knot Floer homology that will be used frequently in the subsequent sections.

Let $K$ be a knot in a rational homology $3$-sphere $Y$. Reversing the orientations of both $Y$ and $K$ yields the reverse of the mirror, denoted $-K\subset -Y$. Given any ${\rm Spin^c}$ structure $\mathfrak{s}\in {\rm Spin^c}(Y)\cong {\rm Spin^c}(-Y)$, there is a filtered chain homotopy equivalence
\begin{equation}\label{mrdual}
CFK^{\infty}(-Y,-K,J\mathfrak{s}) \simeq CFK^{\infty}(Y,K,\mathfrak{s})^{\ast},
\end{equation}
where $CFK^{\infty}(Y,K,\mathfrak{s})^{\ast}$ denotes the dual complex
\[
{\rm Hom}_{\mathbb{F}[U,U^{-1}]}\!\left(CFK^{\infty}(Y,K,\mathfrak{s}),\,\mathbb{F}[U,U^{-1}]\right),
\]
and $J$ is the conjugation action on ${\rm Spin^c}(Y)$. See~\cite[Section~3.5]{OSk}.

In~\cite[Section~2.2]{WY}, the \textit{middle relative ${\rm Spin^c}$ structure} was introduced for each $\mathfrak{s}\in {\rm Spin^c}(Y)$, defined as the unique relative ${\rm Spin^c}$ structure $\xi^0_{\mathfrak{s}}\in \underline{{\rm Spin^c}}(Y,K,\mathfrak{s})$ such that
\[
V_{\xi^0_{\mathfrak{s}}}(Y,K)=H_{\xi^0_{\mathfrak{s}}}(Y,K),
\]
where $H$ is a family of invariants defined similarly to $V$; see~\cite[Section~5.1]{WY} for the precise definition. Several useful properties of the middle relative ${\rm Spin^c}$ structure were also established there.

\begin{lem}[{\cite[Proposition~2.6]{WY}}]\label{Amdrspc}
Suppose $K$ is a knot in a rational homology $3$-sphere $Y$. Then, for any $\rm Spin^c$ structure $\mathfrak{s} \in \rm Spin^c(Y)$, the Alexander grading of the middle relative $\rm Spin^c$ structure $\xi^0_{\mathfrak{s}}$ associated to $\mathfrak{s}$ is given by
\[
A_{Y,K}(\xi^0_{\mathfrak{s}})
=
\frac{1}{2}d(Y,\mathfrak{s})
-
\frac{1}{2}d(Y,\mathfrak{s}+PD[K]).
\]
\end{lem}

\begin{lem}[{\cite[Corollary~2.8]{WY}}]\label{nugr}
Suppose $K$ is a knot in a rational homology $3$-sphere $Y$.   Then, for any $\rm Spin^c$ structure $\mathfrak{s} \in \rm Spin^c(Y)$, we have
\[
\nu^+_{\mathfrak{s}}(Y,K) \geq A_{Y,K}(\xi^0_{\mathfrak{s}}).
\]
\end{lem}

\noindent For simplicity, when $Y$ and $K$ are understood, we write
\[
r_{\mathfrak{s}} = A_{Y,K}(\xi^0_{\mathfrak{s}}).
\]

A simple computation yields the following symmetry in the Alexander gradings of the middle relative ${\rm Spin^c}$ structures.

\begin{lem}\label{mrAsym}
Suppose $K$ is a knot in a rational homology $3$-sphere $Y$.  Then, for any pair of conjugate $\rm Spin^c$ structures $\mathfrak{s}, J\mathfrak{s} \in \rm Spin^c(Y) \cong \rm Spin^c(-Y)$, we have
\[
-A_{Y,K}(\xi^0_{\mathfrak{s}})
=
A_{-Y,-K}(\xi^0_{J\mathfrak{s}}).
\]
\end{lem}
\begin{proof}
By Lemma~\ref{Amdrspc}, we have
\begin{align*}
A_{-Y,-K}(\xi^0_{J\mathfrak{s}})
&=\frac{1}{2}d(-Y,J\mathfrak{s})-\frac{1}{2}d(-Y,J\mathfrak{s}+PD[-K])\\
&=-\frac{1}{2}d(Y,\mathfrak{s})+\frac{1}{2}d(Y,\mathfrak{s}+PD[K])\\
&=-A_{Y,K}(\xi^0_{\mathfrak{s}}),
\end{align*}
where we used the symmetries $d(-Y,\mathfrak{s})=-d(Y,\mathfrak{s})$ and $d(Y,J\mathfrak{s})=d(Y,\mathfrak{s})$.
\end{proof}

\section{Subadditivity and additivity of $\nu^+$ and $V$-invariants}
\label{(Sub)additivities}

\subsection{Subadditivity of $\nu^+$ and $V$-invariants}

The $\nu^+$ and $V$-invariants for knots in rational homology $3$-spheres satisfy subadditivity properties, analogous to the case of knots in $S^3$.

\begin{prop}[Subadditivity of $V$]\label{Vsubadditivity}
Suppose $K_1$ and $K_2$ are knots in rational homology $3$-spheres $Y_1$ and $Y_2$, respectively. Then, for any relative $\rm Spin^c$ structures
\[
\xi_1 \in \underline{\rm Spin^c}(Y_1,K_1)
\qquad\text{ and }\qquad
\xi_2 \in \underline{\rm Spin^c}(Y_2,K_2),
\]
we have
\[
V_{\xi_1\#\xi_2}(Y_1\# Y_2,\,K_1\# K_2)
\leq
V_{\xi_1}(Y_1,K_1)+V_{\xi_2}(Y_2,K_2).
\]
\end{prop}

\begin{proof}
For $i=1,2$, set
\[
n_i=\max\bigl\{{\rm gr}(x)\colon x \in H_{\ast}(A^-_{\xi_i}(K_i)) \text{ is homogeneous and non-torsion}\bigr\}.
\]
Choose a cycle $a_i\in A^-_{\xi_i}(K_i)$ such that $[a_i]\in H_{\ast}(A^-_{\xi_i}(K_i))$ is homogeneous and non-torsion, and
\[
{\rm gr}(a_i)=n_i.
\]
Then $v^-_{\xi_i}(a_i)$ is a cycle in $B^-_{\xi_i}(K_i)$, and
\[
[v^-_{\xi_i}(a_i)]\in H_{\ast}(B^-_{\xi_i}(K_i))
\]
is homogeneous and non-torsion.

Since $B^-_{\xi}(K)\cong CF^-(Y,G_{Y,K}(\xi))$, the K\"unneth formula~\cite[Theorem~6.2]{OSpa} gives
\[
CF^-(Y_1\# Y_2,\mathfrak{s}_1\# \mathfrak{s}_2)\simeq CF^-(Y_1,\mathfrak{s}_1)\otimes CF^-(Y_2,\mathfrak{s}_2),
\]
and hence
\[
B^-_{\xi_1\#\xi_2}(K_1\# K_2)\simeq B^-_{\xi_1}(K_1)\otimes B^-_{\xi_2}(K_2).
\]
Therefore,
\[
v^-_{\xi_1}(a_1)\otimes v^-_{\xi_2}(a_2)\in B^-_{\xi_1\#\xi_2}(K_1\#K_2)
\]
is a cycle whose homology class is homogeneous and non-torsion.

By~\cite[Lemma~3.8]{Rao} and~\cite[Theorem~6.2, Corollary~6.3]{OSpa}, we have
\begin{align*}
A_{Y_1\#Y_2,K_1\#K_2}\bigl(v^-_{\xi_1}(a_1)\otimes v^-_{\xi_2}(a_2)\bigr)
&=A_{Y_1,K_1}\bigl(v^-_{\xi_1}(a_1)\bigr)+A_{Y_2,K_2}\bigl(v^-_{\xi_2}(a_2)\bigr)\\
&\leq A_{Y_1,K_1}(\xi_1)+A_{Y_2,K_2}(\xi_2)\\
&=A_{Y_1\#Y_2,K_1\#K_2}(\xi_1\#\xi_2),
\end{align*}
and
\[
{\rm gr}\bigl(v^-_{\xi_1}(a_1)\otimes v^-_{\xi_2}(a_2)\bigr)
={\rm gr}\bigl(v^-_{\xi_1}(a_1)\bigr)+{\rm gr}\bigl(v^-_{\xi_2}(a_2)\bigr)
=n_1+n_2.
\]
Thus
\[
v^-_{\xi_1}(a_1)\otimes v^-_{\xi_2}(a_2)\in A^-_{\xi_1\#\xi_2}(K_1\#K_2).
\]
Moreover, it is a cycle, and its homology class
\[
\bigl[v^-_{\xi_1}(a_1)\otimes v^-_{\xi_2}(a_2)\bigr]\in H_{\ast}\bigl(A^-_{\xi_1\#\xi_2}(K_1\#K_2)\bigr)
\]
is homogeneous and non-torsion. Let
\[
n_3=\max\bigl\{{\rm gr}(x)\colon x\in H_{\ast}(A^-_{\xi_1\#\xi_2}(K_1\#K_2)) \text{ is homogeneous and non-torsion}\bigr\}.
\]
Then $n_3\geq n_1+n_2$, and therefore
\begin{align*}
V_{\xi_1\#\xi_2}(Y_1\#Y_2,K_1\#K_2)
&=\frac{1}{2}\bigl(d(Y_1\#Y_2,\mathfrak{s}_1\#\mathfrak{s}_2)-n_3\bigr)\\
&\leq \frac{1}{2}\bigl(d(Y_1,\mathfrak{s}_1)-n_1\bigr)+\frac{1}{2}\bigl(d(Y_2,\mathfrak{s}_2)-n_2\bigr)\\
&=V_{\xi_1}(Y_1,K_1)+V_{\xi_2}(Y_2,K_2),
\end{align*}
where $\mathfrak{s}_i=G_{Y_i,K_i}(\xi_i)$ for $i=1,2$, and we used the additivity of the $d$-invariant under connected sum~\cite[Theorem~4.3]{ozsvath2003absolutely}.
\end{proof}

\begin{thm}[Subadditivity of $\nu^+$]
\label{thm:nuSubadditivity}
Suppose $K_1$ and $K_2$ are knots in rational homology $3$-spheres $Y_1$ and $Y_2$, respectively. Then, for any $\rm Spin^c$ structures $\mathfrak{s}_1\in \rm Spin^c(Y_1)$ and $\mathfrak{s}_2\in \rm Spin^c(Y_2)$, we have
\[
\nu^+_{\mathfrak{s}_1\#\mathfrak{s}_2}(Y_1\#Y_2,\,K_1\#K_2)
\leq
\nu^+_{\mathfrak{s}_1}(Y_1,K_1)+\nu^+_{\mathfrak{s}_2}(Y_2,K_2).
\]
\end{thm}

\begin{proof}
For $i=1,2$, choose a relative $\rm Spin^c$ structure
\[
\xi_i\in \underline{\rm Spin^c}(Y_i,K_i,\mathfrak{s}_i)
\]
such that
\[
A_{Y_i,K_i}(\xi_i)=\nu^+_{\mathfrak{s}_i}(Y_i,K_i).
\]
Then $V_{\xi_i}(Y_i,K_i)=0$. By Proposition~\ref{Vsubadditivity} and the non-negativity of the $V$-invariants, it follows that
\[
V_{\xi_1\#\xi_2}(Y_1\#Y_2,K_1\#K_2)=0.
\]
Therefore,
\begin{align*}
\nu^+_{\mathfrak{s}_1\#\mathfrak{s}_2}(Y_1\#Y_2,K_1\#K_2)
&\leq A_{Y_1\#Y_2,K_1\#K_2}(\xi_1\#\xi_2)\\
&=A_{Y_1,K_1}(\xi_1)+A_{Y_2,K_2}(\xi_2)\\
&=\nu^+_{\mathfrak{s}_1}(Y_1,K_1)+\nu^+_{\mathfrak{s}_2}(Y_2,K_2),
\end{align*}
as desired.
\end{proof}

\subsection{Additivity of $\nu^+$-invariants}

We begin by introducing a simple class of full knot Floer complexes $CFK^{\infty}$ that behaves well under tensor products (equivalently, under connected sum of knots).

\begin{defn}\label{def:localtrivial}
Let $K\subset Y$ be a knot in a rational homology $3$-sphere. We say that $K$ is \textit{locally trivial} with respect to a ${\rm Spin^c}$ structure $\mathfrak{s}\in {\rm Spin^c}(Y)$ if there is a filtered chain homotopy equivalence
\begin{equation}\label{localtrivial}
CFK^{\infty}(Y,K,\mathfrak{s}) \simeq CFK^{\infty}(S^3,U)[r_{\mathfrak{s}}] \oplus A,
\end{equation}
where $U$ denotes the unknot in $S^3$, $[r_{\mathfrak{s}}]$ denotes an Alexander grading shift, and $A$ is an acyclic complex. If $K$ is locally trivial for every ${\rm Spin^c}$ structure in ${\rm Spin^c}(Y)$, then we say that $K$ is \textit{totally locally trivial}.
\end{defn}

\begin{rem}
Equation~\eqref{localtrivial} means that there exists a filtered basis for $CFK^{\infty}(Y,K,\mathfrak{s})$ containing an element $x$ that generates the homology $H_{\ast}(CFK^{\infty}(Y,K,\mathfrak{s}))$ and splits off as a direct summand of $CFK^{\infty}(Y,K,\mathfrak{s})$. We call such an $x$ a \textit{distinguished generator}. Note that the distinguished generator is supported at the filtration level where $V=H=0$, and hence
\[
A_{Y,K}(x)=r_{\mathfrak{s}}=A_{Y,K}(\xi^0_{\mathfrak{s}}).
\]
\end{rem}

\begin{rem}\label{rem:localtrivial_nu}
If $K\subset Y$ is locally trivial with respect to a ${\rm Spin^c}$ structure $\mathfrak{s}$, then
\[
\nu^+_{\mathfrak{s}}(Y,K)=r_{\mathfrak{s}}.
\]
\end{rem}

The following proposition gives a necessary and sufficient condition for $\nu^+_{\mathfrak{s}}$ to be additive under connected sum with an arbitrary knot.

\begin{prop}\label{nuadditivity}
Suppose $K_1\subset Y_1$ is a knot in a rational homology $3$-sphere and $\mathfrak{s}_1\in \rm Spin^c(Y_1)$. Then the equality
\[
\nu^+_{\mathfrak{s}_1\#\mathfrak{s}_2}(Y_1\#Y_2,\,K_1\#K_2)
=
\nu^+_{\mathfrak{s}_1}(Y_1,K_1)+\nu^+_{\mathfrak{s}_2}(Y_2,K_2)
\]
holds for every knot $K_2\subset Y_2$ in a rational homology $3$-sphere and every $\rm Spin^c$ structure $\mathfrak{s}_2\in \rm Spin^c(Y_2)$ if and only if $K_1$ is locally trivial with respect to $\mathfrak{s}_1$.
\end{prop}

\begin{proof}
Assume first that $K_1$ is locally trivial with respect to $\mathfrak{s}_1$. Then
\[
\nu^+_{\mathfrak{s}_1}(Y_1,K_1)=A_{Y_1,K_1}(\xi^0_{\mathfrak{s}_1})=r_{\mathfrak{s}_1}.
\]
By the K\"unneth formula~\cite[Theorem~5.1]{OSr} and~\cite[Theorem~3.7, Lemma~3.8]{Rao}, we have
\begin{align*}
CFK^{\infty}(Y_1\#Y_2,K_1\#K_2,\mathfrak{s}_1\#\mathfrak{s}_2)
&\simeq CFK^{\infty}(Y_1,K_1,\mathfrak{s}_1)\otimes CFK^{\infty}(Y_2,K_2,\mathfrak{s}_2)\\
&\simeq \bigl(CFK^{\infty}(S^3,U)[r_{\mathfrak{s}_1}]\oplus A\bigr)\otimes CFK^{\infty}(Y_2,K_2,\mathfrak{s}_2)\\
&\simeq CFK^{\infty}(Y_2,K_2,\mathfrak{s}_2)[r_{\mathfrak{s}_1}] \oplus A',
\end{align*}
where $A':=A\otimes CFK^{\infty}(Y_2,K_2,\mathfrak{s}_2)$ is acyclic. Since acyclic summands do not affect the $\nu^+$-invariant, we obtain
\[
\nu^+_{\mathfrak{s}_1\#\mathfrak{s}_2}(Y_1\#Y_2,K_1\#K_2)
=\nu^+_{\mathfrak{s}_2}(Y_2,K_2)+r_{\mathfrak{s}_1}
=\nu^+_{\mathfrak{s}_2}(Y_2,K_2)+\nu^+_{\mathfrak{s}_1}(Y_1,K_1).
\]

Conversely, assume that the above equality holds for all $(Y_2,K_2,\mathfrak{s}_2)$. Take
\[
(Y_2,K_2,\mathfrak{s}_2)=(-Y_1,-K_1,J\mathfrak{s}_1),
\]
obtained by reversing the orientations of $Y_1$ and $K_1$. By the K\"unneth formula and~\eqref{mrdual}, it follows that
\begin{align*}
CFK^{\infty}(Y_1\#-Y_1,K_1\#-K_1,\mathfrak{s}_1\#J\mathfrak{s}_1)
&\simeq CFK^{\infty}(Y_1,K_1,\mathfrak{s}_1)\otimes CFK^{\infty}(-Y_1,-K_1,J\mathfrak{s}_1)\\
&\simeq CFK^{\infty}(Y_1,K_1,\mathfrak{s}_1)\otimes CFK^{\infty}(Y_1,K_1,\mathfrak{s}_1)^{\ast}\\
&\simeq CFK^{\infty}(S^3,U)\oplus A'',
\end{align*}
for some acyclic complex $A''$. Here the last equivalence follows from~\cite[Lemma~2.18]{Zemcsi} and~\cite[Lemma~3.3]{HHci}. Moreover, there is no Alexander grading shift, since the distinguished generator can only occur in Alexander grading
\begin{align*}
A_{Y_1\#-Y_1,K_1\#-K_1}(\xi^0_{\mathfrak{s}_1\#J\mathfrak{s}_1})
&=\frac{1}{2}d(Y_1\#-Y_1,\mathfrak{s}_1\#J\mathfrak{s}_1)
-\frac{1}{2}d(Y_1\#-Y_1,\mathfrak{s}_1\#J\mathfrak{s}_1+PD[K_1\#-K_1])\\
&=\frac{1}{2}d(Y_1,\mathfrak{s}_1)-\frac{1}{2}d(Y_1,J\mathfrak{s}_1)-\frac{1}{2}d(Y_1,\mathfrak{s}_1+PD[K_1])+\frac{1}{2}d(Y_1,J\mathfrak{s}_1-PD[K_1])\\
&=0.
\end{align*}
Thus,
\[
\nu^+_{\mathfrak{s}_1\#J\mathfrak{s}_1}(Y_1\#-Y_1,K_1\#-K_1)=0.
\]
We claim that
\[
\nu^+_{\mathfrak{s}_1}(Y_1,K_1)=A_{Y_1,K_1}(\xi^0_{\mathfrak{s}_1})
\qquad\text{ and }\qquad
\nu^+_{J\mathfrak{s}_1}(-Y_1,-K_1)=A_{-Y_1,-K_1}(\xi^0_{J\mathfrak{s}_1}).
\]
Otherwise, Lemma~\ref{nugr} would imply
\[
\nu^+_{\mathfrak{s}_1}(Y_1,K_1) > A_{Y_1,K_1}(\xi^0_{\mathfrak{s}_1})
\qquad\text{ or }\qquad
\nu^+_{J\mathfrak{s}_1}(-Y_1,-K_1) > A_{-Y_1,-K_1}(\xi^0_{J\mathfrak{s}_1}).
\]
By Lemma~\ref{mrAsym}, this would imply
\begin{align*}
\nu^+_{\mathfrak{s}_1}(Y_1,K_1)+\nu^+_{J\mathfrak{s}_1}(-Y_1,-K_1)
&>
A_{Y_1,K_1}(\xi^0_{\mathfrak{s}_1})+A_{-Y_1,-K_1}(\xi^0_{J\mathfrak{s}_1})\\
&=0
=\nu^+_{\mathfrak{s}_1\#J\mathfrak{s}_1}(Y_1\#-Y_1,K_1\#-K_1),
\end{align*}
which contradicts the additivity assumption for $\nu^+$ applied to $K_1$ and $-K_1$.
This proves the claim. Consequently,
\[
V_{\xi^0_{\mathfrak{s}_1}}(Y_1,K_1)=0
\qquad\text{ and }\qquad
V_{\xi^0_{J\mathfrak{s}_1}}(-Y_1,-K_1)=0.
\]
By~\cite[Proposition~6.1]{WY} and \cite[Proposition~3.11]{Hom}, it follows that $K_1$ is locally trivial with respect to $\mathfrak{s}_1$.
\end{proof}

Our main theorem relies on the additivity of the $\nu^+$-invariant. Using the Heegaard Floer $\tau$-invariants~\cite{Rao,HR23}, one can prove an analogue of Theorem~\ref{KcUY1}. However, we emphasize that $\nu^+$ gives a lower bound on slice genus that is at least as strong as that given by $\tau$, and in general strictly stronger, as shown in~\cite[Theorem~1]{HW}. Thus, $\nu^+$-sharp knots form a larger family than $\tau$-sharp knots, which may be defined analogously.

\section{Additivity of the rational slice genus}
\label{Proof of main results}

We first establish the following inequality.

\begin{lem}\label{lem:subadditivity}
Let $K$ and $K'$ be knots in rational homology $3$-spheres $Y$ and $Y'$, respectively. Then
\begin{equation*}\label{rsgie}
g_{(Y\#Y') \times I}(K\#K') \leq g_{Y \times I}(K) + g_{Y' \times I}(K').
\end{equation*}
\end{lem}

\begin{proof}
In~\cite[Lemma~5.1]{NV}, the authors construct a rational Seifert surface $S_0$ for $K\#K'$ from rational Seifert surfaces $S$ for $K$ and $S'$ for $K'$.

The same construction extends to the $4$-dimensional setting. Suppose $F$ and $F'$ are Seifert framed rational slice surfaces for $K$ and $K'$, respectively. We construct a Seifert framed rational slice surface $F_0$ for $K\#K'$ by attaching bands along arcs near the boundary, exactly as in the construction of $S_0$. In particular, since this operation on $\partial F$ and $\partial F'$ is the same as the operation on $\partial S$ and $\partial S'$, we have $\partial F_0=\partial S_0$. Therefore, $F_0$ is Seifert framed, and its genus satisfies
\[
g(F_0)=g(F)+g(F').
\]
Taking $F$ and $F'$ to be genus-minimizing completes the proof.
\end{proof}

In order to prove Theorem~\ref{KcUY1}, we will also need the following facts about totally locally trivial knots.

\begin{lem}\label{tltnusym}
Suppose $K$ is a totally locally trivial knot in a rational homology $3$-sphere $Y$. Then $-K$ is also totally locally trivial in $-Y$, and
\[
\nu^+(Y,K)=\nu^+(-Y,-K).
\]
\end{lem}

\begin{proof}\pushQED{\qed}
Suppose $K \subset Y$ is a totally locally trivial knot in a rational homology $3$-sphere. Set
\[
\mathcal{A}_{Y,K}=\{A_{Y,K}(\xi^0_{\mathfrak{s}})\mid \mathfrak{s}\in \rm Spin^c(Y)\},
\]
and define
\[
A_{\max}(Y,K)=\max \mathcal{A}_{Y,K}, \qquad A_{\min}(Y,K)=\min \mathcal{A}_{Y,K}.
\]
Since $K$ is totally locally trivial, we have
\[
\nu^+(Y,K)=A_{\max}(Y,K).
\]

For any $\rm Spin^c$ structure $\mathfrak{s}\in \rm Spin^c(Y)$, consider the middle relative $\rm Spin^c$ structures $\xi^0_{\mathfrak{s}}$ and $\xi^0_{J\mathfrak{s}-PD[K]}$. By Lemma~\ref{Amdrspc},
\begin{align*}
A_{Y,K}(\xi^0_{J\mathfrak{s}-PD[K]})
&=\frac{1}{2}d(Y,J\mathfrak{s}-PD[K])-\frac{1}{2}d(Y,J\mathfrak{s}-PD[K]+PD[K])\\
&=\frac{1}{2}d(Y,\mathfrak{s}+PD[K])-\frac{1}{2}d(Y,\mathfrak{s})\\
&=-A_{Y,K}(\xi^0_{\mathfrak{s}}),
\end{align*}
where we used $d(Y,J\mathfrak{s})=d(Y,\mathfrak{s})$. It follows that $\mathcal{A}_{Y,K}$ is symmetric about $0$, and hence
\begin{equation}\label{rlspcsym}
A_{\max}(Y,K)=-A_{\min}(Y,K).
\end{equation}

By~\eqref{mrdual}, the fact that $K\subset Y$ is totally locally trivial implies that $-K\subset -Y$ is also totally locally trivial. Moreover, Lemma~\ref{mrAsym} gives
\[
\mathcal{A}_{Y,K}=-\mathcal{A}_{-Y,-K}
\qquad\text{and}\qquad
A_{\max}(Y,K)=-A_{\min}(-Y,-K).
\]
Using~\eqref{rlspcsym} for $(-Y,-K)$, we obtain
\[
A_{\max}(-Y,-K)=-A_{\min}(-Y,-K)=A_{\max}(Y,K).
\]
Therefore,
\[
\nu^+(-Y,-K)=A_{\max}(-Y,-K)=A_{\max}(Y,K)=\nu^+(Y,K).\qedhere
\]
\end{proof}

We now prove Theorem~\ref{KcUY1}, whose statement we recall for convenience.

\begingroup
\renewcommand{\thethm}{1.3}
\begin{thm}
\addtocounter{thm}{-1}
Let $K$ and $K'$ be $\nu^+$-sharp knots in rational homology $3$-spheres $Y$ and $Y'$ respectively. If, in addition, $K'$ is totally locally trivial, then $K\#K'$ is $\nu^+$-sharp and
\[
g_{(Y\# Y')\times I}(K\# K') = g_{Y\times I}(K) + g_{Y'\times I}(K').
\]
\end{thm}
\endgroup


\begin{proof}
By assumption and Lemma~\ref{tltnusym}, we have
\[
g_{Y \times I}(K)=\max\{\nu^+(Y,K), \nu^+(-Y,-K)\}
\qquad\text{ and }\qquad
g_{Y'\times I}(K')=\nu^+(Y',K')=\nu^+(-Y',-K').
\]

We first consider the case where
\[
g_{Y \times I}(K)=\nu^+(Y,K).
\]
By Lemma~\ref{lem:subadditivity} and Theorem~\ref{nuleqprsg}, we have
\begin{align*}
\nu^+(Y\#Y',K\#K')
&=\max_{\mathfrak{o}\in {\rm Spin^c}(Y\#Y')}\nu^+_{\mathfrak{o}}(Y\#Y',K\#K')\\
&\le g_{(Y\#Y') \times I}(K\#K')\\
&\le g_{Y \times I}(K)+g_{Y' \times I}(K').
\end{align*}
On the other hand, under the standard identification
\[
{\rm Spin^c}(Y\#Y')\cong {\rm Spin^c}(Y)\times {\rm Spin^c}(Y'),
\]
every $\mathfrak{o}\in{\rm Spin^c}(Y\#Y')$ can be written as $\mathfrak{o}=\mathfrak{s}\#\mathfrak{t}$.
Proposition~\ref{nuadditivity} therefore implies that
\[
\nu^+_{\mathfrak{o}}(Y\#Y',K\#K')
=\nu^+_{\mathfrak{s}}(Y,K)+\nu^+_{\mathfrak{t}}(Y',K').
\]
Taking the maximum on both sides, we obtain
\[
\nu^+(Y\#Y',K\#K')=\nu^+(Y,K)+\nu^+(Y',K').
\]
Since
\[
g_{Y \times I}(K)=\nu^+(Y,K)
\qquad\text{ and }\qquad g_{Y' \times I}(K')=\nu^+(Y',K'),
\]
it follows that
\[
g_{Y \times I}(K)+g_{Y' \times I}(K')
=\nu^+(Y\#Y', K\#K')
\le g_{(Y\#Y') \times I}(K\#K')
\le g_{Y \times I}(K)+g_{Y' \times I}(K').
\]
Hence,
\[
g_{(Y\#Y') \times I}(K\#K')=g_{Y \times I}(K)+g_{Y' \times I}(K').
\]

If instead $g_{Y \times I}(K)=\nu^+(-Y,-K)$, then applying the same argument to $(-Y\#-Y',-K\#-K')$ yields the corresponding equality for $(Y\#Y',K\#K')$.

Finally, we show that $K\#K'\subset Y\#Y'$ is also $\nu^+$-sharp. Since, by Lemma~\ref{tltnusym}, $-K' \subset -Y'$ is also totally locally trivial, Proposition~\ref{nuadditivity} gives
\[
\nu^+(-Y\#-Y',-K\#-K')=\nu^+(-Y,-K)+\nu^+(-Y',-K').
\]
Therefore,
\begin{align*}
g_{(Y\#Y') \times I}(K\#K')
&=g_{Y \times I}(K)+g_{Y' \times I}(K')\\
&=\max\{\nu^+(Y,K),\nu^+(-Y,-K)\}+g_{Y' \times I}(K')\\
&=\max\{\nu^+(Y,K)+\nu^+(Y',K'),\ \nu^+(-Y,-K)+\nu^+(-Y',-K')\}\\
&=\max\{\nu^+(Y\#Y',K\#K'),\ \nu^+(-Y\#-Y',-K\#-K')\},
\end{align*}
which concludes the proof.
\end{proof}

By definition, Lemma \ref{tltnusym} and \cite[Theorem 1.4]{WY}, Floer simple knots are totally locally trivial and $\nu^+$-sharp. We now show that the unknot in a rational homology $3$-sphere is also totally locally trivial and $\nu^+$-sharp.

\begin{lem}\label{UinQHS}
If $U$ is the unknot in a rational homology $3$-sphere $Y'$, then $U$ is totally locally trivial and $\nu^+$-sharp. More precisely, 
for any $\mathfrak{t}\in {\rm Spin^c}(Y')$, there is a filtered chain homotopy equivalence
\[
CFK^{\infty}(Y',U,\mathfrak{t}) \simeq CFK^{\infty}(S^3,U) \oplus A,
\]
where $CFK^{\infty}(S^3,U)$ denotes the full knot Floer complex of the unknot in $S^3$ and $A$ is an acyclic complex. In addition, every generator of $CFK^{\infty}(Y',U,\mathfrak{t})$ has Alexander grading $0$.
\end{lem}
\begin{proof}
Since $g_{Y' \times I}(U)=0$, Theorem~\ref{nuleqprsg} implies that for any $\mathfrak{t}\in{\rm Spin^c}(Y')$,
\[
\nu_{\mathfrak{t}}^+(Y',U) \leq \nu^+(Y',U) \leq g_{Y' \times I}(U)=0.
\]
Using Lemma~\ref{Amdrspc} and Lemma~\ref{nugr}, we obtain that for any $\mathfrak{t} \in {\rm Spin^c}(Y')$,
\[
A_{Y',U}(\xi^0_{\mathfrak{t}})
=\frac{1}{2}d(Y',\mathfrak{t})-\frac{1}{2}d(Y',\mathfrak{t}+PD[U])=0 \leq \nu^+_{\mathfrak{t}}(Y',U),
\]
since $U$ is null-homologous. This shows that $\nu^+(Y',U)=\nu_{\mathfrak{t}}^+(Y',U)=0$ for all $\mathfrak{t}\in{\rm Spin^c}(Y')$. 

By the definition of $\nu_{\mathfrak{t}}^+$, it follows that
\[ 
V_{\xi^0_{\mathfrak{t}}}(Y',U)=0
\]
for all $\mathfrak{t}\in{\rm Spin^c}(Y')$. Similarly, since $g_{(-Y')\times I}(-U)=0$, we also have
\[\nu^+(-Y',-U)=\nu^+_{\mathfrak{t}}(-Y',-U)=0 \qquad \text{and} \qquad
V_{\xi^0_{\mathfrak{t}}}(-Y',-U)=0
\]
for all $\mathfrak{t}\in{\rm Spin^c}(Y')$. Thus, 
\[\max\{\nu^+(Y',U),\nu^+(-Y',-U)\}=g_{Y' \times I}(U)=0,\]
and the conclusion that $U$ is totally locally trivial now follows from~\cite[Proposition~6.1]{WY} and~\cite[Proposition~3.11]{Hom}.

Finally, knot Floer homology detects the (rational) Seifert genus~\cite[Theorem~1.1]{Ni} \cite[Theorem 2.2]{NW0} and satisfies the symmetry ~\cite[Section~2.2]{HL} 
\[
\widehat{HFK}(Y,K,r)\cong \widehat{HFK}(Y,K,-r),
\]
where
\[
\widehat{HFK}(Y,K,r)=\mathop{\bigoplus}\limits_{\{\xi \in \underline{{\rm Spin^c}}(Y,K)\mid A_{Y,K}(\xi)=r\}} \widehat{HFK}(Y,K,\xi).
\]
Then we conclude that all generators of $CFK^{\infty}(Y',U,\mathfrak{t})$ have Alexander grading $0$ by $g_{Y'}(U)=0$.
\end{proof}


\section{Examples of $\nu^+$-sharp knots and Floer simple knots}
\label{sec:examples_nu_sharp_floer_simple}

\subsection{$\nu^+$-sharp knots in $S^3$}
There are many $\nu^+$-sharp knots in $S^3$. In this subsection, we present several families of knots with this property.

\subsubsection{Examples with $|\tau(K)|=g_4(K)$}

The $\tau$-invariant is a knot concordance invariant derived from knot Floer homology. Ozsv\'ath and Szab\'o~\cite{OSf}, together with~\cite[Proposition~2.3]{HW}, showed that
\[
|\tau(K)| \leq \max\{\nu^+(K), \nu^+(-K)\} \leq g_4(K).
\]
Therefore, if a knot $K\subset S^3$ satisfies $|\tau(K)|=g_4(K)$, then $K$ is $\nu^+$-sharp. We now record some examples of knots with this property.

\begin{itemize}
\item[(a)] $L$-space knots, that is, knots admitting a positive surgery to an $L$-space.
Examples include Berge knots~\cite{Berge}, some twisted torus knots~\cite[Theorem~1.1]{Vafaee}, some Montesinos knots~\cite[Theorem~1]{BaMo,LiMo}, some $(1,1)$ knots~\cite[Theorem~1.2]{GLV}, all algebraic knots~\cite[Theorem~2]{GorNe}, and some satellite knots~\cite{Hom11,HLV15,HRW}.

\item[(b)] Quasipositive knots~\cite{Pla,Hed10}, that is, knots that are closures of braids consisting of arbitrary conjugates of positive generators. Quasipositive knots are equivalent to transverse $\mathbb{C}$-knots, that is, knots arising as the transverse intersection of a complex curve with the unit sphere $S^3\subset \mathbb{C}^2$. They include positive braids, positive knots, and strongly quasipositive knots. See~\cite{Hed10} for more details.

\item[(c)] The non-positively twisted, positive-clasped Whitehead double of any knot $K$ with $\tau(K)>0$~\cite[Theorem~1.5]{Hed09}.

\item[(d)] Knots that are squeezed between a positive torus knot and the unknot~\cite[Proposition~4.1]{FLL24}, that is, knots arising as slices of genus-minimizing, oriented, connected, smooth cobordisms between a positive torus knot and the unknot.

\item[(e)] Linear combinations of torus knots of the form $aT(p,q)\#-bT(p',q')$, for certain choices of positive integers $p,q,p',q',a,b$~\cite[Theorems~2, 3, 16]{LiVa}.

\item[(f)] The untwisted satellites of any knot $K$ with $\tau(K)=g_4(K)>0$ by $L$-space satellite operators~\cite[Theorem~1.5]{CZZ25}, and the non-negative twisted satellites of any knot $K$ with $\tau(K)=g_4(K)>0$ by any $L$-space satellite operator with minimal wrapping number~\cite[Proposition~1.8]{CZZ25}.

\item[(g)] Any connected sum of knots satisfying $\tau(K)=g_4(K)$ retains this property, since $\tau$ is additive under connected sum.
\end{itemize}

\subsubsection{Examples with $|\tau(K)|<\max\{\nu^+(K), \nu^+(-K)\}=g_4(K)$}
\label{egtaulesnueqg4}
In this subsection, we present some examples of knots satisfying
\[
|\tau(K)|<\max\{\nu^+(K), \nu^+(-K)\}=g_4(K).
\]

\noindent \textit{Livingston and Van Cott's examples:}
Livingston and Van Cott~\cite[Theorems~2 and~16]{LiVa} show that the following linear combinations of torus knots
\begin{itemize}
\item[(a)] $aT(p,qr)\#-bT(q,pr)$ with $q>p>0$, $q/(q-p)>r>0$ and $a>b>0$;
\item[(b)] $aT(2,10r+1)\#-bT(3,6r+1)$ with $a>b>0$ and $r>0$;
\item[(c)] $aT(2,10r+3)\#-bT(3,6r+2)$ with $a>b>0$ and $r>0$;
\item[(d)] $aT(2,10r+1)\#-bT(4,4r+1)$ with $a>b>0$ and $r>0$
\end{itemize}
satisfy
\[
|\tau(K)|<\max_{t\in(0,1]} \frac{|\Upsilon_K(t)|}{t}=g_4(K),
\]
where $\Upsilon_K(t)$ is also a knot concordance invariant derived from knot Floer homology. For any knot $K\subset S^3$,
\[
|\tau(K)| \leq \max_{t\in(0,1]} \frac{|\Upsilon_K(t)|}{t} \leq \max\{\nu^+(K), \nu^+(-K)\} \leq g_4(K).
\]
Therefore, these knots are $\nu^+$-sharp.

\vspace{1.5mm}

\noindent \textit{Hom and Wu's examples, and cabling:}
Hom and the second author~\cite{HW} constructed $\nu^+$-sharp examples with $|\tau(K)|<g_4(K)$, namely the knot
$T_{2,5}\#2T_{2,3}\#-T_{2,3;2,5}$, where $T_{2,3;2,5}$ denotes the $(2,5)$-cable of $T_{2,3}$. They then applied a cabling formula for $\nu^+$~\cite[Propositions~3.5--3.6]{HW} (see also~\cite[Corollary~1.2]{Wu16}) to show that for any positive integer $p$, there exists a knot $K$ with
\[
|\tau(K)|+p \leq \max\{ \nu^+(K), \nu^+(-K)\}=g_4(K),
\]
by considering the $(p,3p-1)$-cable of $T_{2,5}\#2T_{2,3}\#-T_{2,3;2,5}$.

In addition, Sato extended the cabling formula to any positive cable.

\begin{thm}[{\cite[Corollary 1.4]{Sato}}]\label{nucabling}
Given a knot $K \subset S^3$, if $\nu^+(K)=g_4(K)$, then for any coprime $p,q>0$, we have
\[
\nu^+(K_{p,q})=g_4(K_{p,q})=pg_4(K)+\frac{(p-1)(q-1)}{2}.
\]
\end{thm}

\noindent Applying Theorem~\ref{nucabling} together with Hom's cabling formula for $\tau$~\cite[Theorem~1]{Hom14}, it follows that this phenomenon is preserved under positive cabling. Namely, if $K$ is $\nu^+$-sharp and satisfies $|\tau(K)|<g_4(K)$, then so does any positive cable of $K$.

\subsection{$\nu^+$-sharp knots in rational homology $3$-spheres}

As stated in Theorem~\ref{KcUY1} (and related results such as Corollary~\ref{4gad}), we can immediately construct examples of $\nu^+$-sharp knots by taking connected sums with unknots or Floer simple knots. In addition, there are examples not obtained in this manner, namely $L$-space knots in $L$-spaces.

As in the case of $S^3$, a knot $K$ in an $L$-space $Y$ that admits a positive $L$-space surgery, equivalently, an \textit{$L$-space knot}, satisfies
\begin{equation}\label{Lskn}
\nu^+(Y,K)=\max\{\nu^+(Y,K), \nu^+(-Y,-K)\}=g_{Y}(K)=g_{Y \times I}(K).
\end{equation}
As in the case of $L$-space knots in $S^3$, the knot Floer complex of an $L$-space knot in an $L$-space has a staircase shape~\cite[Lemma~3.2]{Ras17}. Combined with the fact that knot Floer homology detects knot genus~\cite[Theorem~1.1]{Ni}~\cite[Theorem~2.2]{NW0}, this implies that $\nu^+(Y,K)=g_Y(K)$, and hence~\eqref{Lskn}.

Work on $L$-space knots in $L$-spaces other than $S^3$ is relatively sparse. Here we list a few examples, focusing primarily on knots in lens spaces.
\begin{itemize}
\item[(a)] Berge-Gabai knots~\cite{Gab90} (or their mirror images) in lens spaces. These include all torus knots in a solid torus of a lens space.
\item[(b)] All $1$-bridge braids in lens spaces~\cite[Theorem~1.4]{GLV}.
\item[(c)] Some $(1,1)$ knots in lens spaces. In~\cite[Theorem~1.2]{GLV}, Greene, Lewallen, and Vafaee characterize the $(1,1)$ $L$-space knots in lens spaces.
\end{itemize}
There is overlap among (a)--(c). In particular,
\[
\text{(a)}\smallsetminus\{\text{torus knots}\}\subset \text{(b)}\subset \text{(c)}.
\]

\subsection{Floer simple knots}
\label{Floer simple knots}

Below, we list some examples of Floer simple knots.
\begin{itemize}
\item[(a)] The unknot is the unique null-homologous Floer simple knot in an $L$-space.
\item[(b)] Simple knots in lens spaces~\cite[Section~2.1]{Ras07}.
\item[(c)] Let $Y$ be a compact, connected, oriented $3$-manifold with torus boundary, and let $\mathcal{L}(Y)$ denote the set of $L$-space filling slopes of $Y$. For a slope $\alpha$, write $Y(\alpha)$ for the closed manifold obtained by Dehn filling $Y$ along $\alpha$, and let $K_\alpha \subset Y(\alpha)$ be the core of the filling solid torus. If $\alpha \in \mathcal{L}^{\circ}(Y)$, the interior of $\mathcal{L}(Y)$, then $K_{\alpha}$ is Floer simple in $Y(\alpha)$~\cite[Corollary~3.6]{Ras17}.

In particular, the dual knot to $\alpha$-surgery on an $L$-space knot $K \subset S^3$ is Floer simple whenever $\alpha \in (2g(K)-1,\infty)$, where $g(K)$ denotes the Seifert genus of $K$ (see also~\cite{Gre15,Ras07,Hed11}).

\item[(d)] Any connected sum of Floer simple knots is Floer simple.
\end{itemize}

\bibliographystyle{alpha}
\bibliography{tex}

@article {BaMo,
    AUTHOR = {Baker, Kenneth L. and Moore, Allison H.},
     TITLE = {Montesinos knots, {H}opf plumbings, and {L}-space surgeries},
   JOURNAL = {J. Math. Soc. Japan},
  FJOURNAL = {Journal of the Mathematical Society of Japan},
    VOLUME = {70},
      YEAR = {2018},
    NUMBER = {1},
     PAGES = {95--110},
      ISSN = {0025-5645,1881-1167},
   MRCLASS = {57M25 (57M27)},
  MRNUMBER = {3750269},
MRREVIEWER = {Masakazu\ Teragaito},
       DOI = {10.2969/jmsj/07017484},
       URL = {https://doi.org/10.2969/jmsj/07017484},
}

@article {Kronheimer-Mrowka:1993,
    AUTHOR = {Kronheimer, Peter B. and Mrowka, Tomasz S.},
     TITLE = {Gauge theory for embedded surfaces. {I}},
   JOURNAL = {Topology},
  FJOURNAL = {Topology. An International Journal of Mathematics},
    VOLUME = {32},
      YEAR = {1993},
    NUMBER = {4},
     PAGES = {773--826},
      ISSN = {0040-9383},
   MRCLASS = {57R57 (57N13 57R40 57R55 58D29)},
  MRNUMBER = {1241873},
MRREVIEWER = {Ronald\ J.\ Stern},
       DOI = {10.1016/0040-9383(93)90051-V},
       URL = {https://doi.org/10.1016/0040-9383(93)90051-V},
}

@article {Berge,
    AUTHOR = {Berge, John},
     TITLE = {Some knots with surgeries yielding lens spaces},
   JOURNAL = {preprint arXiv:1802.09722},
      YEAR = {1990}
}

@article {Raoux-Hedden:2023,
    AUTHOR = {Hedden, Matthew and Raoux, Katherine},
     TITLE = {A 4-dimensional rational genus bound},
   JOURNAL = {preprint arXiv:2308.16853},
      YEAR = {2023}
}

@article {CZZ25,
    AUTHOR = {Chen, Daren and Zemke, Ian and Zhou,  Hugo},
     TITLE = {Applications of the {L}-space satellite formula},
   JOURNAL = {preprint arXiv:2509.20288},
      YEAR = {2025},
}

@article {DNPR,
    AUTHOR = {Davis, Christopher W. and Nagel, Matthias and Park, JungHwan
              and Ray, Arunima},
     TITLE = {Concordance of knots in {$S^1\times S^2$}},
   JOURNAL = {J. Lond. Math. Soc. (2)},
  FJOURNAL = {Journal of the London Mathematical Society. Second Series},
    VOLUME = {98},
      YEAR = {2018},
    NUMBER = {1},
     PAGES = {59--84},
      ISSN = {0024-6107,1469-7750},
   MRCLASS = {57M27},
  MRNUMBER = {3847232},
MRREVIEWER = {Allison\ N.\ Miller},
       DOI = {10.1112/jlms.12125},
       URL = {https://doi.org/10.1112/jlms.12125},
}

@article {FLL24,
    AUTHOR = {Feller, Peter and Lewark, Lukas and Lobb, Andrew},
     TITLE = {Squeezed knots},
   JOURNAL = {Quantum Topol.},
  FJOURNAL = {Quantum Topology},
    VOLUME = {16},
      YEAR = {2025},
    NUMBER = {4},
     PAGES = {831--865},
      ISSN = {1663-487X,1664-073X},
   MRCLASS = {57K10 (57K18)},
  MRNUMBER = {4958851},
       DOI = {10.4171/qt/187},
       URL = {https://doi.org/10.4171/qt/187},
}

@article {Gab90,
    AUTHOR = {Gabai, David},
     TITLE = {{$1$}-bridge braids in solid tori},
   JOURNAL = {Topology Appl.},
  FJOURNAL = {Topology and its Applications},
    VOLUME = {37},
      YEAR = {1990},
    NUMBER = {3},
     PAGES = {221--235},
      ISSN = {0166-8641,1879-3207},
   MRCLASS = {57M25 (57M15 57N10 57R30 57R95)},
  MRNUMBER = {1082933},
MRREVIEWER = {Cameron\ McA.\ Gordon},
       DOI = {10.1016/0166-8641(90)90021-S},
       URL = {https://doi.org/10.1016/0166-8641(90)90021-S},
}

@article {GLV,
    AUTHOR = {Greene, Joshua Evan and Lewallen, Sam and Vafaee, Faramarz},
     TITLE = {{$(1,1)$} {L}-space knots},
   JOURNAL = {Compos. Math.},
  FJOURNAL = {Compositio Mathematica},
    VOLUME = {154},
      YEAR = {2018},
    NUMBER = {5},
     PAGES = {918--933},
      ISSN = {0010-437X,1570-5846},
   MRCLASS = {57M27},
  MRNUMBER = {3798589},
MRREVIEWER = {Kazuhiro\ Ichihara},
       DOI = {10.1112/S0010437X17007989},
       URL = {https://doi.org/10.1112/S0010437X17007989},
}

@article {GorNe,
    AUTHOR = {Gorsky, Eugene and N\'emethi, Andr\'as},
     TITLE = {Links of plane curve singularities are {$L$}-space links},
   JOURNAL = {Algebr. Geom. Topol.},
  FJOURNAL = {Algebraic \& Geometric Topology},
    VOLUME = {16},
      YEAR = {2016},
    NUMBER = {4},
     PAGES = {1905--1912},
      ISSN = {1472-2747,1472-2739},
   MRCLASS = {57M27 (14H20)},
  MRNUMBER = {3546454},
MRREVIEWER = {Meirav\ Amram},
       DOI = {10.2140/agt.2016.16.1905},
       URL = {https://doi.org/10.2140/agt.2016.16.1905},
}

@article {HRW,
    AUTHOR = {Hanselman, Jonathan and Rasmussen, Jacob and Watson, Liam},
     TITLE = {Bordered {F}loer homology for manifolds with torus boundary
              via immersed curves},
   JOURNAL = {J. Amer. Math. Soc.},
  FJOURNAL = {Journal of the American Mathematical Society},
    VOLUME = {37},
      YEAR = {2024},
    NUMBER = {2},
     PAGES = {391--498},
      ISSN = {0894-0347,1088-6834},
   MRCLASS = {57K31 (57K18 57M50)},
  MRNUMBER = {4695506},
MRREVIEWER = {Francesco\ Lin},
       DOI = {10.1090/jams/1029},
       URL = {https://doi.org/10.1090/jams/1029},
}

@article {Hed09,
    AUTHOR = {Hedden, Matthew},
     TITLE = {On knot {F}loer homology and cabling. {II}},
   JOURNAL = {Int. Math. Res. Not. IMRN},
  FJOURNAL = {International Mathematics Research Notices. IMRN},
      YEAR = {2009},
    NUMBER = {12},
     PAGES = {2248--2274},
      ISSN = {1073-7928,1687-0247},
   MRCLASS = {57M27 (57R58)},
  MRNUMBER = {2511910},
       DOI = {10.1093/imrn/rnp015},
       URL = {https://doi.org/10.1093/imrn/rnp015},
}

@article {Hed10,
    AUTHOR = {Hedden, Matthew},
     TITLE = {Notions of positivity and the {O}zsv\'ath-{S}zab\'o{}
              concordance invariant},
   JOURNAL = {J. Knot Theory Ramifications},
  FJOURNAL = {Journal of Knot Theory and its Ramifications},
    VOLUME = {19},
      YEAR = {2010},
    NUMBER = {5},
     PAGES = {617--629},
      ISSN = {0218-2165,1793-6527},
   MRCLASS = {57M27 (57M25 57N70)},
  MRNUMBER = {2646650},
       DOI = {10.1142/S0218216510008017},
       URL = {https://doi.org/10.1142/S0218216510008017},
}

@article {HL,
    AUTHOR = {Hedden, Matthew and Levine, Adam Simon},
     TITLE = {A surgery formula for knot {F}loer homology},
   JOURNAL = {Quantum Topol.},
  FJOURNAL = {Quantum Topology},
    VOLUME = {15},
      YEAR = {2024},
    NUMBER = {2},
     PAGES = {229--336},
      ISSN = {1663-487X,1664-073X},
   MRCLASS = {57K18},
  MRNUMBER = {4725826},
       DOI = {10.4171/qt/188},
       URL = {https://doi.org/10.4171/qt/188},
}

@article {Hom11,
    AUTHOR = {Hom, Jennifer},
     TITLE = {A note on cabling and {$L$}-space surgeries},
   JOURNAL = {Algebr. Geom. Topol.},
  FJOURNAL = {Algebraic \& Geometric Topology},
    VOLUME = {11},
      YEAR = {2011},
    NUMBER = {1},
     PAGES = {219--223},
      ISSN = {1472-2747,1472-2739},
   MRCLASS = {57M27 (57R58)},
  MRNUMBER = {2764041},
       DOI = {10.2140/agt.2011.11.219},
       URL = {https://doi.org/10.2140/agt.2011.11.219},
}

@article {Hom14,
    AUTHOR = {Hom, Jennifer},
     TITLE = {Bordered {H}eegaard {F}loer homology and the tau-invariant of
              cable knots},
   JOURNAL = {J. Topol.},
  FJOURNAL = {Journal of Topology},
    VOLUME = {7},
      YEAR = {2014},
    NUMBER = {2},
     PAGES = {287--326},
      ISSN = {1753-8416,1753-8424},
   MRCLASS = {57M27 (57M25 57R58)},
  MRNUMBER = {3217622},
MRREVIEWER = {Arunima\ Ray},
       DOI = {10.1112/jtopol/jtt030},
       URL = {https://doi.org/10.1112/jtopol/jtt030},
}

@article {Hom,
    AUTHOR = {Hom, Jennifer},
     TITLE = {A survey on {H}eegaard {F}loer homology and concordance},
   JOURNAL = {J. Knot Theory Ramifications},
  FJOURNAL = {Journal of Knot Theory and its Ramifications},
    VOLUME = {26},
      YEAR = {2017},
    NUMBER = {2},
     PAGES = {1740015, 24},
      ISSN = {0218-2165,1793-6527},
   MRCLASS = {57R58 (57M25 57N70)},
  MRNUMBER = {3604497},
MRREVIEWER = {Se-Goo\ Kim},
       DOI = {10.1142/S0218216517400156},
       URL = {https://doi.org/10.1142/S0218216517400156},
}

@article {HLV15,
    AUTHOR = {Hom, Jennifer and Lidman, Tye and Vafaee, Faramarz},
     TITLE = {Berge-{G}abai knots and {L}-space satellite operations},
   JOURNAL = {Algebr. Geom. Topol.},
  FJOURNAL = {Algebraic \& Geometric Topology},
    VOLUME = {14},
      YEAR = {2014},
    NUMBER = {6},
     PAGES = {3745--3763},
      ISSN = {1472-2747,1472-2739},
   MRCLASS = {57M25 (57M27 57R58)},
  MRNUMBER = {3302978},
MRREVIEWER = {Kimihiko\ Motegi},
       DOI = {10.2140/agt.2014.14.3745},
       URL = {https://doi.org/10.2140/agt.2014.14.3745},
}

@article {HW,
    AUTHOR = {Hom, Jennifer and Wu, Zhongtao},
     TITLE = {Four-ball genus bounds and a refinement of the {O}zsv\'ath-{S}zab\'o{} tau invariant},
   JOURNAL = {J. Symplectic Geom.},
  FJOURNAL = {The Journal of Symplectic Geometry},
    VOLUME = {14},
      YEAR = {2016},
    NUMBER = {1},
     PAGES = {305--323},
      ISSN = {1527-5256,1540-2347},
   MRCLASS = {57M27 (57M25 57R58)},
  MRNUMBER = {3523259},
MRREVIEWER = {Nikolai\ N.\ Saveliev},
       DOI = {10.4310/JSG.2016.v14.n1.a12},
       URL = {https://doi.org/10.4310/JSG.2016.v14.n1.a12},
}

@article {KR,
    AUTHOR = {Klug, Michael R. and Ruppik, Benjamin M.},
     TITLE = {Deep and shallow slice knots in 4-manifolds},
   JOURNAL = {Proc. Amer. Math. Soc. Ser. B},
  FJOURNAL = {Proceedings of the American Mathematical Society. Series B},
    VOLUME = {8},
      YEAR = {2021},
     PAGES = {204--218},
      ISSN = {2330-1511},
   MRCLASS = {57K40 (57K10)},
  MRNUMBER = {4273166},
MRREVIEWER = {Allison\ N.\ Miller},
       DOI = {10.1090/bproc/89},
       URL = {https://doi.org/10.1090/bproc/89},
}

@article {LiMo,
    AUTHOR = {Lidman, Tye and Moore, Allison H.},
     TITLE = {Pretzel knots with {$L$}-space surgeries},
   JOURNAL = {Michigan Math. J.},
  FJOURNAL = {Michigan Mathematical Journal},
    VOLUME = {65},
      YEAR = {2016},
    NUMBER = {1},
     PAGES = {105--130},
      ISSN = {0026-2285,1945-2365},
   MRCLASS = {57M27 (57M25 57R58)},
  MRNUMBER = {3466818},
MRREVIEWER = {Arunima\ Ray},
       URL = {http://projecteuclid.org/euclid.mmj/1457101813},
}

@article {LiVa,
    AUTHOR = {Livingston, Charles and Van Cott, Cornelia A.},
     TITLE = {The four-genus of connected sums of torus knots},
   JOURNAL = {Math. Proc. Cambridge Philos. Soc.},
  FJOURNAL = {Mathematical Proceedings of the Cambridge Philosophical
              Society},
    VOLUME = {164},
      YEAR = {2018},
    NUMBER = {3},
     PAGES = {531--550},
      ISSN = {0305-0041,1469-8064},
   MRCLASS = {57M25},
  MRNUMBER = {3784267},
MRREVIEWER = {Makoto\ Ozawa},
       DOI = {10.1017/S0305004117000342},
       URL = {https://doi.org/10.1017/S0305004117000342},
}

@article {NOPP,
    AUTHOR = {Nagel, Matthias and Orson, Patrick and Park, JungHwan and Powell, Mark},
     TITLE = {Smooth and topological almost concordance},
   JOURNAL = {Int. Math. Res. Not. IMRN},
  FJOURNAL = {International Mathematics Research Notices. IMRN},
      YEAR = {2019},
    NUMBER = {23},
     PAGES = {7324--7355},
      ISSN = {1073-7928,1687-0247},
   MRCLASS = {57K18 (57K10)},
  MRNUMBER = {4039014},
MRREVIEWER = {Daniel\ Silver},
       DOI = {10.1093/imrn/rnx338},
       URL = {https://doi.org/10.1093/imrn/rnx338},
}

@article {Ni,
    AUTHOR = {Ni, Yi},
     TITLE = {Link {F}loer homology detects the {T}hurston norm},
   JOURNAL = {Geom. Topol.},
  FJOURNAL = {Geometry \& Topology},
    VOLUME = {13},
      YEAR = {2009},
    NUMBER = {5},
     PAGES = {2991--3019},
      ISSN = {1465-3060,1364-0380},
   MRCLASS = {57M27 (57R58)},
  MRNUMBER = {2546619},
       DOI = {10.2140/gt.2009.13.2991},
       URL = {https://doi.org/10.2140/gt.2009.13.2991},
}

@article {NV,
    AUTHOR = {Ni, Yi and Vafaee, Faramarz},
     TITLE = {Null surgery on knots in {L}-spaces},
   JOURNAL = {Trans. Amer. Math. Soc.},
  FJOURNAL = {Transactions of the American Mathematical Society},
    VOLUME = {372},
      YEAR = {2019},
    NUMBER = {12},
     PAGES = {8279--8306},
      ISSN = {0002-9947,1088-6850},
   MRCLASS = {57K18},
  MRNUMBER = {4029697},
MRREVIEWER = {Duncan\ McCoy},
       DOI = {10.1090/tran/7510},
       URL = {https://doi.org/10.1090/tran/7510},
}

@article {NW0,
    AUTHOR = {Ni, Yi and Wu, Zhongtao},
     TITLE = {Heegaard {F}loer correction terms and rational genus bounds},
   JOURNAL = {Adv. Math.},
  FJOURNAL = {Advances in Mathematics},
    VOLUME = {267},
      YEAR = {2014},
     PAGES = {360--380},
      ISSN = {0001-8708,1090-2082},
   MRCLASS = {57R58 (57M25 57M27)},
  MRNUMBER = {3269182},
MRREVIEWER = {Tye\ Lidman},
       DOI = {10.1016/j.aim.2014.09.006},
       URL = {https://doi.org/10.1016/j.aim.2014.09.006},
}

@article {NW1,
    AUTHOR = {Ni, Yi and Wu, Zhongtao},
     TITLE = {Cosmetic surgeries on knots in {$S^3$}},
   JOURNAL = {J. Reine Angew. Math.},
  FJOURNAL = {Journal f\"ur die Reine und Angewandte Mathematik. [Crelle's
              Journal]},
    VOLUME = {706},
      YEAR = {2015},
     PAGES = {1--17},
      ISSN = {0075-4102,1435-5345},
   MRCLASS = {57M25 (57M27)},
  MRNUMBER = {3393360},
MRREVIEWER = {Bruno\ P.\ Zimmermann},
       DOI = {10.1515/crelle-2013-0067},
       URL = {https://doi.org/10.1515/crelle-2013-0067},
}

@article {OSf,
    AUTHOR = {Ozsv\'ath, Peter and Szab\'o, Zolt\'an},
     TITLE = {Knot {F}loer homology and the four-ball genus},
   JOURNAL = {Geom. Topol.},
  FJOURNAL = {Geometry and Topology},
    VOLUME = {7},
      YEAR = {2003},
     PAGES = {615--639},
      ISSN = {1465-3060,1364-0380},
   MRCLASS = {57R58 (57M25 57M27)},
  MRNUMBER = {2026543},
MRREVIEWER = {Stanislav\ Jabuka},
       DOI = {10.2140/gt.2003.7.615},
       URL = {https://doi.org/10.2140/gt.2003.7.615},
}

@article {OSk,
    AUTHOR = {Ozsv\'ath, Peter and Szab\'o, Zolt\'an},
     TITLE = {Holomorphic disks and knot invariants},
   JOURNAL = {Adv. Math.},
  FJOURNAL = {Advances in Mathematics},
    VOLUME = {186},
      YEAR = {2004},
    NUMBER = {1},
     PAGES = {58--116},
      ISSN = {0001-8708,1090-2082},
   MRCLASS = {57M27 (57R58)},
  MRNUMBER = {2065507},
MRREVIEWER = {Stanislav\ Jabuka},
       DOI = {10.1016/j.aim.2003.05.001},
       URL = {https://doi.org/10.1016/j.aim.2003.05.001},
}

@article {OSr,
    AUTHOR = {Ozsv\'ath, Peter  and Szab\'o, Zolt\'an},
     TITLE = {Knot {F}loer homology and rational surgeries},
   JOURNAL = {Algebr. Geom. Topol.},
  FJOURNAL = {Algebraic \& Geometric Topology},
    VOLUME = {11},
      YEAR = {2011},
    NUMBER = {1},
     PAGES = {1--68},
      ISSN = {1472-2747,1472-2739},
   MRCLASS = {57R58 (57M25 57M27)},
  MRNUMBER = {2764036},
MRREVIEWER = {Hans\ U.\ Boden},
       DOI = {10.2140/agt.2011.11.1},
       URL = {https://doi.org/10.2140/agt.2011.11.1},
}

@article {Pla,
    AUTHOR = {Plamenevskaya, Olga},
     TITLE = {Bounds for the {T}hurston-{B}ennequin number from {F}loer
              homology},
   JOURNAL = {Algebr. Geom. Topol.},
  FJOURNAL = {Algebraic \& Geometric Topology},
    VOLUME = {4},
      YEAR = {2004},
     PAGES = {399--406},
      ISSN = {1472-2747,1472-2739},
   MRCLASS = {57R17 (57M27 57R58)},
  MRNUMBER = {2077671},
MRREVIEWER = {Stanislav\ Jabuka},
       DOI = {10.2140/agt.2004.4.399},
       URL = {https://doi.org/10.2140/agt.2004.4.399},
}

@article {Rao,
    AUTHOR = {Raoux, Katherine},
     TITLE = {{$\tau$}-invariants for knots in rational homology spheres},
   JOURNAL = {Algebr. Geom. Topol.},
  FJOURNAL = {Algebraic \& Geometric Topology},
    VOLUME = {20},
      YEAR = {2020},
    NUMBER = {4},
     PAGES = {1601--1640},
      ISSN = {1472-2747,1472-2739},
   MRCLASS = {57K16 (57R58)},
  MRNUMBER = {4127080},
MRREVIEWER = {Christopher\ William\ Davis},
       DOI = {10.2140/agt.2020.20.1601},
       URL = {https://doi.org/10.2140/agt.2020.20.1601},
}

@article {Rasmussen:2010,
    AUTHOR = {Rasmussen, Jacob},
     TITLE = {Khovanov homology and the slice genus},
   JOURNAL = {Invent. Math.},
  FJOURNAL = {Inventiones Mathematicae},
    VOLUME = {182},
      YEAR = {2010},
    NUMBER = {2},
     PAGES = {419--447},
      ISSN = {0020-9910,1432-1297},
   MRCLASS = {57M27},
  MRNUMBER = {2729272},
MRREVIEWER = {William\ D.\ Gillam},
       DOI = {10.1007/s00222-010-0275-6},
       URL = {https://doi.org/10.1007/s00222-010-0275-6},
}

@article {Murasugi:1965,
    AUTHOR = {Murasugi, Kunio},
     TITLE = {On a certain numerical invariant of link types},
   JOURNAL = {Trans. Amer. Math. Soc.},
  FJOURNAL = {Transactions of the American Mathematical Society},
    VOLUME = {117},
      YEAR = {1965},
     PAGES = {387--422},
      ISSN = {0002-9947,1088-6850},
   MRCLASS = {55.20},
  MRNUMBER = {171275},
MRREVIEWER = {R.\ H.\ Fox},
       DOI = {10.2307/1994215},
       URL = {https://doi.org/10.2307/1994215},
}

@article {Ras17,
    AUTHOR = {Rasmussen, Jacob and Rasmussen, Sarah Dean},
     TITLE = {Floer simple manifolds and {L}-space intervals},
   JOURNAL = {Adv. Math.},
  FJOURNAL = {Advances in Mathematics},
    VOLUME = {322},
      YEAR = {2017},
     PAGES = {738--805},
      ISSN = {0001-8708,1090-2082},
   MRCLASS = {57R58 (57M27)},
  MRNUMBER = {3720808},
MRREVIEWER = {Matthew\ Stoffregen},
       DOI = {10.1016/j.aim.2017.10.014},
       URL = {https://doi.org/10.1016/j.aim.2017.10.014},
}

@article {Sato,
    AUTHOR = {Sato, Kouki},
     TITLE = {A full-twist inequality for the {$\nu^+$}-invariant},
   JOURNAL = {Topology Appl.},
  FJOURNAL = {Topology and its Applications},
    VOLUME = {245},
      YEAR = {2018},
     PAGES = {113--130},
      ISSN = {0166-8641,1879-3207},
   MRCLASS = {57M27},
  MRNUMBER = {3823992},
MRREVIEWER = {Shida\ Wang},
       DOI = {10.1016/j.topol.2018.06.010},
       URL = {https://doi.org/10.1016/j.topol.2018.06.010},
}

@article {Vafaee,
    AUTHOR = {Vafaee, Faramarz},
     TITLE = {On the knot {F}loer homology of twisted torus knots},
   JOURNAL = {Int. Math. Res. Not. IMRN},
  FJOURNAL = {International Mathematics Research Notices. IMRN},
      YEAR = {2015},
    NUMBER = {15},
     PAGES = {6516--6537},
      ISSN = {1073-7928,1687-0247},
   MRCLASS = {57M27},
  MRNUMBER = {3384486},
MRREVIEWER = {Fatemeh\ Douroudian},
       DOI = {10.1093/imrn/rnu130},
       URL = {https://doi.org/10.1093/imrn/rnu130},
}

@article {Wu16,
    AUTHOR = {Wu, Zhongtao},
     TITLE = {A cabling formula for the {$\nu^+$} invariant},
   JOURNAL = {Proc. Amer. Math. Soc.},
  FJOURNAL = {Proceedings of the American Mathematical Society},
    VOLUME = {144},
      YEAR = {2016},
    NUMBER = {9},
     PAGES = {4089--4098},
      ISSN = {0002-9939,1088-6826},
   MRCLASS = {57M27 (57M25)},
  MRNUMBER = {3513564},
MRREVIEWER = {William\ W.\ Menasco},
       DOI = {10.1090/proc/13029},
       URL = {https://doi.org/10.1090/proc/13029},
}

@article {WY,
    AUTHOR = {Wu, Zhongtao and Yang, Jingling},
     TITLE = {Rational genus and {H}eegaard {F}loer homology},
   JOURNAL = {preprint arXiv:2307.06807},
      YEAR = {2023}
}

@article {Feller-Park:2021,
    AUTHOR = {Feller, Peter and Park, JungHwan},
     TITLE = {Genus one cobordisms between torus knots},
   JOURNAL = {Int. Math. Res. Not. IMRN},
  FJOURNAL = {International Mathematics Research Notices. IMRN},
      YEAR = {2021},
    NUMBER = {1},
     PAGES = {523--550},
      ISSN = {1073-7928,1687-0247},
   MRCLASS = {57K10 (57R75)},
  MRNUMBER = {4198504},
MRREVIEWER = {Christopher\ William\ Davis},
       DOI = {10.1093/imrn/rnaa027},
       URL = {https://doi.org/10.1093/imrn/rnaa027},
}

@article {BCG:2017,
    AUTHOR = {Bodn\'ar, J\'ozsef and Celoria, Daniele and Golla, Marco},
     TITLE = {A note on cobordisms of algebraic knots},
   JOURNAL = {Algebr. Geom. Topol.},
  FJOURNAL = {Algebraic \& Geometric Topology},
    VOLUME = {17},
      YEAR = {2017},
    NUMBER = {4},
     PAGES = {2543--2564},
      ISSN = {1472-2747,1472-2739},
   MRCLASS = {57M25 (14B05 14B07 57M27 57R58)},
  MRNUMBER = {3686406},
MRREVIEWER = {Yuanan\ Diao},
       DOI = {10.2140/agt.2017.17.2543},
       URL = {https://doi.org/10.2140/agt.2017.17.2543},
}

@article {ozsvath2003absolutely,
    AUTHOR = {Ozsv\'ath, Peter and Szab\'o, Zolt\'an},
     TITLE = {Absolutely graded {F}loer homologies and intersection forms
              for four-manifolds with boundary},
   JOURNAL = {Adv. Math.},
  FJOURNAL = {Advances in Mathematics},
    VOLUME = {173},
      YEAR = {2003},
    NUMBER = {2},
     PAGES = {179--261},
      ISSN = {0001-8708,1090-2082},
   MRCLASS = {57R58 (57M27)},
  MRNUMBER = {1957829},
MRREVIEWER = {Jacob\ Andrew\ Rasmussen},
       DOI = {10.1016/S0001-8708(02)00030-0},
       URL = {https://doi.org/10.1016/S0001-8708(02)00030-0},
}

@article{FKQS2026,
  title  = {Gordian distance between local knots in 3-manifolds},
  author = {Friedl, Stefan and Kalelkar, Tejas and Quintanilha, Jos{\'e} Pedro and Shah, Tanushree},
  year   = {2026},
}

@article {Gre15,
    AUTHOR = {Greene, Joshua Evan},
     TITLE = {L-space surgeries, genus bounds, and the cabling conjecture},
   JOURNAL = {J. Differential Geom.},
  FJOURNAL = {Journal of Differential Geometry},
    VOLUME = {100},
      YEAR = {2015},
    NUMBER = {3},
     PAGES = {491--506},
      ISSN = {0022-040X,1945-743X},
   MRCLASS = {57M25 (57M27)},
  MRNUMBER = {3352796},
MRREVIEWER = {Daniel\ Ruberman},
       URL = {http://projecteuclid.org/euclid.jdg/1432842362},
}

@book {Rasmussen:2003,
    AUTHOR = {Rasmussen, Jacob Andrew},
     TITLE = {Floer homology and knot complements},
      NOTE = {Thesis (Ph.D.)--Harvard University},
 PUBLISHER = {ProQuest LLC, Ann Arbor, MI},
      YEAR = {2003},
     PAGES = {126},
      ISBN = {978-0496-39374-9},
   MRCLASS = {99-05},
  MRNUMBER = {2704683},
       URL =
              {http://gateway.proquest.com/openurl?url_ver=Z39.88-2004&rft_val_fmt=info:ofi/fmt:kev:mtx:dissertation&res_dat=xri:pqdiss&rft_dat=xri:pqdiss:3091665},
}

@article {Ras07,
    AUTHOR = {Rasmussen, Jacob},
     TITLE = {Lens space surgeries and L-space homology spheres},
   JOURNAL = {preprint arXiv:arXiv:0710.2531},
      YEAR = {2007}
}

@article {Hed11,
    AUTHOR = {Hedden, Matthew},
     TITLE = {On {F}loer homology and the {B}erge conjecture on knots
              admitting lens space surgeries},
   JOURNAL = {Trans. Amer. Math. Soc.},
  FJOURNAL = {Transactions of the American Mathematical Society},
    VOLUME = {363},
      YEAR = {2011},
    NUMBER = {2},
     PAGES = {949--968},
      ISSN = {0002-9947,1088-6850},
   MRCLASS = {57M25 (57M27 57R58)},
  MRNUMBER = {2728591},
MRREVIEWER = {Adam\ M.\ Lowrance},
       DOI = {10.1090/S0002-9947-2010-05117-7},
       URL = {https://doi.org/10.1090/S0002-9947-2010-05117-7},
}

@article {OzSz:2004,
    AUTHOR = {Ozsv\'ath, Peter and Szab\'o, Zolt\'an},
     TITLE = {Holomorphic disks and topological invariants for closed
              three-manifolds},
   JOURNAL = {Ann. of Math. (2)},
  FJOURNAL = {Annals of Mathematics. Second Series},
    VOLUME = {159},
      YEAR = {2004},
    NUMBER = {3},
     PAGES = {1027--1158},
      ISSN = {0003-486X,1939-8980},
   MRCLASS = {57M27 (32Q65 57R58)},
  MRNUMBER = {2113019},
MRREVIEWER = {Thomas\ E.\ Mark},
       DOI = {10.4007/annals.2004.159.1027},
       URL = {https://doi.org/10.4007/annals.2004.159.1027},
}

@article {OSpa,
    AUTHOR = {Ozsv\'ath, Peter and Szab\'o, Zolt\'an},
     TITLE = {Holomorphic disks and three-manifold invariants: properties
              and applications},
   JOURNAL = {Ann. of Math. (2)},
  FJOURNAL = {Annals of Mathematics. Second Series},
    VOLUME = {159},
      YEAR = {2004},
    NUMBER = {3},
     PAGES = {1159--1245},
      ISSN = {0003-486X,1939-8980},
   MRCLASS = {57M27 (32Q65 57R58)},
  MRNUMBER = {2113020},
MRREVIEWER = {Thomas\ E.\ Mark},
       DOI = {10.4007/annals.2004.159.1159},
       URL = {https://doi.org/10.4007/annals.2004.159.1159},
}

@article {Zemcsi,
    AUTHOR = {Zemke, Ian},
     TITLE = {Connected sums and involutive knot {F}loer homology},
   JOURNAL = {Proc. Lond. Math. Soc. (3)},
  FJOURNAL = {Proceedings of the London Mathematical Society. Third Series},
    VOLUME = {119},
      YEAR = {2019},
    NUMBER = {1},
     PAGES = {214--265},
      ISSN = {0024-6115,1460-244X},
   MRCLASS = {57M27 (57R56)},
  MRNUMBER = {3957835},
MRREVIEWER = {Sergiy\ Koshkin},
       DOI = {10.1112/plms.12227},
       URL = {https://doi.org/10.1112/plms.12227},
}

@incollection {HHci,
    AUTHOR = {Hendricks, Kristen and Hom, Jennifer},
     TITLE = {A note on knot concordance and involutive knot {F}loer
              homology},
 BOOKTITLE = {Breadth in contemporary topology},
    SERIES = {Proc. Sympos. Pure Math.},
    VOLUME = {102},
     PAGES = {113--118},
 PUBLISHER = {Amer. Math. Soc., Providence, RI},
      YEAR = {2019},
      ISBN = {978-1-4704-4249-1},
   MRCLASS = {57M27 (57M25)},
  MRNUMBER = {3967364},
MRREVIEWER = {Allison\ N.\ Miller},
       DOI = {10.1090/pspum/102/09},
       URL = {https://doi.org/10.1090/pspum/102/09},
}

@article {HR23,
    AUTHOR = {Hedden, Matthew and Raoux, Katherine},
     TITLE = {Knot {F}loer homology and relative adjunction inequalities},
   JOURNAL = {Selecta Math. (N.S.)},
  FJOURNAL = {Selecta Mathematica. New Series},
    VOLUME = {29},
      YEAR = {2023},
    NUMBER = {1},
     PAGES = {Paper No. 7, 48},
      ISSN = {1022-1824,1420-9020},
   MRCLASS = {57K18 (57K10 57K31 57K33 57K41 57R58 57R65)},
  MRNUMBER = {4507976},
MRREVIEWER = {Tye\ Lidman},
       DOI = {10.1007/s00029-022-00810-1},
       URL = {https://doi.org/10.1007/s00029-022-00810-1},
}

\end{document}